\documentclass[11pt,leqno,british]{amsart}
\usepackage{ifpdf}

\usepackage[T1]{fontenc}

\usepackage{eucal}
\usepackage{url}
\usepackage[pagebackref]{hyperref}
\usepackage{amsfonts}
\usepackage{amsthm}
\usepackage{amsmath}
\usepackage{amscd}
\usepackage{amssymb}
\usepackage{graphicx}
\usepackage{rotating}
\usepackage{wasysym}
\usepackage[small,nohug,heads=LaTeX]{diagrams}
\diagramstyle[labelstyle=\scriptstyle]
\usepackage{ulem}
\usepackage{textcomp}
\usepackage{ mathdots }
\usepackage{tikz}
\usetikzlibrary{matrix,arrows,decorations.pathmorphing}

\def\NN {{\mathbb N}}     
\def\RR {{\mathbb R}}     
\def\ZZ {{\mathbb Z}}     

\def\ring#1{\ifmmode \mathaccent'027 #1\else \rm\accent'027 #1\fi}

\def\ul  {\underline}

\def\wt  {\widetilde}

\def\ker {\mathfrak{ker}}

\def\mc {\mathcal}
\def\mk {\mathfrak}

\def\Hom {\mathrm{Hom}}

\def\st {\mathrm{Stab}}

\def \bd {\begin{diagram}}
\def \ed {\end{diagram}}
\def\be  {\begin{eqnarray}}
\def\ee  {\end{eqnarray}}
\def\ben {\begin{eqnarray*}}
\def\een {\end{eqnarray*}}

\def\bpr {\begin{proof}[Proof]}
\def\epr {\end{proof}}
\def\bsp {\begin{split}}
\def\esp {\end{split}}
\def\bprr {\begin{proof}[solution]}
\def\bpru {\begin{proof}[hint]}
\def\bpro {\begin{proof}[answer]}
\def\bcd {\begin{CD}}
\def\ecd {\end{CD}}

\newcommand{\scal}[1]{\left\langle#1\right\rangle}

\newtheorem{theorem}{Theorem}[section]
\newtheorem{lemma}[theorem]{Lemma}
\newtheorem{prop}[theorem]{Proposition}
\newtheorem{coro}[theorem]{Corollary}
\newtheorem{remark}[theorem]{Remark}
\newtheorem{df}[theorem]{Definition}

\newtheorem{conj}[theorem]{Conjecture}
\newtheorem{quest}[theorem]{Question}

\begin{document}

\title[Non-semistable exceptional objects in hereditary categories: some remarks and ...]%
{Non-semistable exceptional objects in hereditary categories: some remarks and conjectures}
\author{George Dimitrov}
\address[Dimitrov]{Universität Wien\\
Oskar-Morgenstern-Platz 1, 1090 Wien\\
Österreich}
\email{george.dimitrov@univie.ac.at}

\author{Ludmil Katzarkov}
\address[Katzarkov]{Universität Wien\\
Oskar-Morgenstern-Platz 1, 1090 Wien\\
Österreich}
\email{lkatzarkov@gmail.com}

\begin{abstract}
In our previous paper we studied non-semistable exceptional objects in hereditary categories and introduced the notion of regularity preserving category, but we obtained quite a few examples of such categories. Certain conditions on the Ext-nontrivial couples (exceptional objects $X,Y\in \mathcal A$ with ${\rm Ext}^1(X,Y)\neq 0$ and ${\rm Ext}^1(Y,X)\neq 0$) were shown to imply regularity-preserving. This paper is a brief review of the previous paper (with emphasis on regularity preserving property) and we add some remarks and conjectures.

It is known that in Dynkin quivers ${\rm Hom}(\rho,\rho')=0$ or ${\rm Ext}^1(\rho,\rho')=0$ for any two exceptional representations. In the present paper we use this property to show that for any Dynkin quiver $Q$ there are no Ext-nontrivial couples in $Rep_k(Q)$, which implies regularity preserving of $Rep_k(Q)$, where $k$ is an algebraically closed field. We study this property in other quivers. In particular in any star quiver with three arms $Q$ for any two exceptional representations $\rho, \rho'$ we have ${\rm Hom}(\rho,\rho')=0$ or ${\rm Ext}^1(\rho,\rho')=0$ provided that $\rho$ or $\rho'$ is a thin representation. In the previous version we asserted  falsely that this holds for any two exceptional representations (without imposing the restriction that one of them is thin) for extended Dynkin quivers $\widetilde{\mathbb E}_6, \widetilde{\mathbb E}_7, \widetilde{\mathbb E}_8 $.  

 \end{abstract}

\maketitle
\setcounter{tocdepth}{2}
\tableofcontents

\section{Introduction}

   T. Bridgeland introduced in his seminal work \cite{Bridg1} the definition of a locally finite stability condition  on a triangulated category $\mc T$,
   motivated by  M. Douglas's notion of $\Pi$-stability. He
       proved that the set of these stability conditions is a complex manifold, denoted by $\st(\mc T)$, on which act the groups $\widetilde{GL}^+(2,\RR)$ and ${\rm Aut}(\mc T)$. To any bounded t-structure of $\mc T$ he assigned a set of stability conditions.

E. Macri showed in \cite{Macri}  that the extension closure of   a full Ext-exceptional collection $\mc E=(E_0,E_1,\dots, E_n)$ in  $\mc T$ is a bounded t-structure. 
The stability conditions obtained from this t-structure together with their  translations by the right action of $\widetilde{GL}^+(2,\RR)$ are referred to as \textit{generated by $\mc E$} \cite{Macri}.  
E. Macr\`i, studying   $ \st(D^b(K(l))$ in  \cite{Macri},     gave  an idea for producing an  exceptional pair generating a given stability condition  $\sigma$ on  $D^b(K(l))$, where $K(l)$ is the $l$-Kronecker quiver. 

 We defined in \cite{DK}  the notion of a \textit{$\sigma$-exceptional collection} (\cite[Definition 3.19]{DK}), so that the full $\sigma$-exceptional collections are exactly the exceptional collections which generate $\sigma$, and  we   focused  on  constructing $\sigma$-exceptional collections from a given $\sigma \in \st(D^b(\mc A))$, where $\mc A$ is a hereditary, $\hom$-finite, abelian category.  We 
 developed tools   for constructing $\sigma$-exceptional collections of length at least three in  $D^b(\mc A)$.  These tools are based on  
   the notion of       \textit{regularity-preserving hereditary category},  introduced in \cite{DK}.
	
 After a detailed study of  the exceptional objects of the quiver $Q_1=$ $\begin{diagram}[size=1em]
       &       &  \circ  &       &        \\
       & \ruTo &         & \luTo &        \\
\circ  & \rTo  &         &       &  \circ
\end{diagram}$   it was shown in \cite{DK} that $Rep_k(Q_1)$ is regularity preserving and  the newly obtained methods for constructing $\sigma$-triples were applied to the case   $\mc A = Rep_k(Q_1)$.   As a result we  obtained        the following theorem:

 \begin{theorem}[\cite{DK}] \label{main theorem for Q_1 in intro} Let $k$ be an algebraically closed field.
 For each $\sigma \in \st(D^b(Rep_k(Q_1)))$ there exists a full
$\sigma$-exceptional collection.
\end{theorem}

 In other words,  all stability conditions on $D^b(Q_1)$ are generated by exceptional collections.  This theorem implies
 that $\st(D^b(Q_1))$ is connected \cite[Corollary 10.2]{DK}.   Theorem \ref{main theorem for Q_1 in intro} and the data about the exceptional collections  
 \cite[Section 2]{DK}  are a basis for proving that $\st(D^b(Rep_k(Q_1)))$ is contractible (done in a subsequent paper).

  One difficulty in the proof of  Theorem \ref{main theorem for Q_1 in intro}  is   due to the
     Ext-nontrivial   couples (Definition \ref{Ext-nontrivial couple}). 	In the present  paper we show that  this difficulty does not arise in $Rep_k(Q)$, where $Q$ is any Dynkin quiver, which motivates us for Conjecture \ref{conj1}.

\vspace{3mm}

\noindent 1.1. We recall  now  in  more detail how the notion regularity preserving category appeared    in \cite{DK}.  
By  $\mc A$ we denote  a $k$-linear  $\hom$-finite hereditary  abelian category, where  $k$ is an algebraically closed field, and we denote $D^b(\mc A)$  by  $\mc T$. We choose $\sigma \in \st(\mc T)$.

In   \cite[Sections 4]{DK} is given an algorithm $\mk{alg}$, which produces a     distinguished triangle $\mk{alg}(R)=\begin{diagram}[size=1.0em] U & \rTo      &    R  &  \rTo     & V  & \rTo     & U[1] \end{diagram}$ with   semi-stable $V$,    from  any unstable  exceptional    object $R \in \mc T$.  The triangle $\mk{alg}(R)$ satisfies the vanishings  $\hom^1(U,U)=\hom^1(V,V)=\hom^*(U,V)=0$   for any unstable exceptional object $R$, provided that the category $\mc A$ has no Ext-nontrivial couples.  When $\mc A$ contains Ext-nontrivial couples,  these vanishings are not guaranteed by $\mk{alg}$.  By definition we call a non-semistable exceptional object $R$ $\sigma$-regular, when these vanishings  hold for $\mk{alg}(R)$, otherwise we call it $\sigma$-irregular.  In particular, if  $\mc A$ has no Ext-nontrivial couples, then all non-semistable objects are $\sigma$-regular.  When $R$ is $\sigma$-regular,  the vanishings in $\mk{alg}(R)$ imply that  for any indecomposable components $S$ and $E$ of $V$  and $U$, respectively,    the  pair $(S,E)$ is exceptional with semistable  first element $S$.    We denote this relation between a $\sigma$-regular object  $R$ and the exceptional pair
   $(S,E)$ by $\bd  R & \rDotsto^{X} & (S,E)  \ed $, where $X$ contains further information (see \cite[Definition 5.2]{DK}).

The $\sigma$-regular objects in turn  are divided into final and non-final as follows.  In each relation  $\bd  R &
\rDotsto & (S,E)  \ed $ the first component $S$ is a semistable exceptional object, and the second is not restricted to be always semistable. If there is such a relation with a non-semistable $E$, then we refer to $R$ as  a \textit{nonfinal} $\sigma$-regular object, otherwise - \textit{final}.  The name nonfinal is justified, when the category $\mc A$ has a specific property called regularity-preserving, defined  as follows: 
 \begin{df} {\rm (}\cite[Definition 6.1]{DK}{\rm ) } \label{RP category} A hereditary abelian category $\mc A$ will be said to be\\ \ul{regularity-preserving}, if for each $\sigma \in \st(D^b(\mc A))$ from the
 the following  data:

 $R\in D^b(\mc A)$ is a $\sigma$-regular object; $ \begin{diagram} R & \rDotsto & (S,E) \end{diagram}$; $E \not \in \sigma^{ss}$ \vspace{2mm}

it follows that  $E$   is  a  $\sigma$-regular object as well.
 \end{df}

   In a regularity-preserving category $\mc A$ the relation  $\bd   & \rDotsto &   \ed $ circumvents the $\sigma$-irregular objects, and  each non-final
   $\sigma$-regular object $R$   generates  a long  sequence\footnote{By ``long'' we mean that it has at least two steps. This sequence  is not uniquely determined
   by $R$.} of the form:  \begin{gather} \label{sequence of cases in intro}  \bd[height=1.5em] R & \rDotsto^{X_1} & (S_1,E_1) & \rMapsto^{proj_2}& E_1 & \rDotsto^{X_2} & (S_2,E_2)  & \rMapsto^{proj_2}& E_2 & \rDotsto^{X_3} & (S_3,E_3)& \rMapsto^{proj_2}&  \dots \\
  &   & \dMapsto^{proj_1} &    &  &    & \dMapsto^{proj_1} &  &  &  & \dMapsto^{proj_1} &  &  & \\
   &  &  S_1              &    &  &    &  S_2              &  &  &  & S_3               &  &  &  \ed. \end{gather}    Furthermore,  after finitely many steps (say $n$) a final $\sigma$-regular object  $E_n$ appears (\cite[Lemma 7.1]{DK})  and then $S_1, S_2$, $\dots$, $S_{n+1},E_{n+1}$($n\geq 1$)  is a sequence of semistable and exceptional objects. The last pair $(S_{n+1},E_{n+1})$ is always exceptional, however  the entire sequence is not always an exceptional collection. 
  The sequences of the form \eqref{sequence of cases in intro} generated by   $\sigma$-regular objects are the main tool  used in \cite[Sections 7,8,9]{DK}  for constructing $\sigma$-exceptional collections.  \vspace{5mm}

 1.2.  We explain now the  known examples of regularity preserving categories.  Apart from its expository aspect,   this paper  enlarges the small list of examples of such categories. 

 In \cite[Section 6]{DK} we  found certain conditions on  the Ext-nontrivial couples of $\mc A$,  called
   \textit{RP property 1 and RP property 2} (see Definition \ref{def of ENC}  below),   which imply regularity-preserving. In particular, non-existing of such couples implies this property. Thus, regularity preserving  property is related to  specific pairwise relations between the exceptional objects of $\mc A$.  
	
	The study of exceptional objects in quivers goes back to \cite{S1}, \cite{S2}, \cite{WCB1}, and to \cite{Ringel0} for more general hereditary categories.   However, to the best of our knowledge, no attention to the Ext-nontrivial couples has been focused. 
		
	It follows from   \cite[Lemma 4.1]{Macri} that there are no
 Ext-nontrivial couples  in $Rep_k(K(l))$ and hence  $Rep_k(K(l))$ is an example of regularity preserving category (see \cite[Appendix B]{DK}).    

 Two more examples  are  the categories  $Rep_k(Q_1)$ and $Rep_k(Q_2)$ (the quiver $Q_2$ is in Fig. \eqref{Q1} in Section 2).  In \cite{DK} is shown that  there are Ext-nontrivial couples in  $Rep_k(Q_1)$ and $Rep_k(Q_2)$, but after computing  the   dimensions of $\Hom(X,Y)$, ${\rm Ext}^1(X,Y)$ for any two exceptional objects $X,Y$ and  by a careful analysis of the results  we showed that 
      RP properties 1 and  2 do hold in both the cases. Furthermore,  the resulting tables of dimensions show  that one  of the  spaces  $\Hom(X,Y)$, ${\rm Ext}^1(X,Y)$   always vanishes.

\vspace{3mm}

1.3  The  new (and easy) examples of regular preserving categories given in the present paper are the categories of representations of  Dynkin quivers (Corollary \ref{Dynkin are regularity preserv}).

 The basic observation  is   that if a quiver $Q$ satisfies    $\Hom_Q(\rho,\rho')=0$ or ${\rm Ext}_Q^1(\rho,\rho')=0$ for any two exceptional representations $\rho,\rho'$, then the dimension vectors of any Ext-nontrivial couple $\{\rho,\rho'\}$  satisfy $\langle \ul{\dim}(\rho) + \ul{\dim}(\rho'),\ul{\dim}(\rho) + \ul{\dim}(\rho') \rangle \leq 0$ (Lemma \ref{lemma for leq 0 and ENC}). This motivates us to study in more  detail the property that  $\Hom(\rho,\rho')=0$ or ${\rm Ext}^1(\rho,\rho')=0$ for   given exceptional representations $\rho,\rho'$ $\in Rep_k(Q)$. In  Section \ref{two examples} we recall some results from \cite{DK} about the exceptional objects of  $Rep_k(Q_1)$, $Rep_k(Q_2)$, in particular  this property holds in $Rep_k(Q_1)$ and  $Rep_k(Q_2)$ for any two exceptional representations (Corollary \ref{RP property 1,2 and.. for Q1} (b)).  An example  with an acyclic quiver where this fails  is obtained  by changing the orientation of the quiver $Q_2$ (see \eqref{non maximal rank}).

In Section \ref{the linear map} is recalled  the definition of the standard differential in the 2-term complex computing $\RR \Hom_Q(\rho,\rho')$  for any two representations   $\rho$, $\rho' \in Rep_k(Q)$, which we denote by  $F_{\rho,\rho'}^Q$. We utilize  this linear map  because the condition that  one of the two spaces  $\Hom(\rho,\rho')$ or ${\rm Ext}^1(\rho,\rho')$ vanishes is the same as the condition that  $F_{\rho,\rho'}^Q$ has maximal rank.   In Sections \ref{without loops}, \ref{star shaped}  we find   conditions which ensure maximality of the rank of  $F_{\rho,\rho'}^Q$.  The  strategy  is to expand  the simple linear-algebraic observations: Lemma  \ref{big lemma} (a),(b) and Lemma \ref{simple lemma} to big enough quivers by using   Corollary \ref{main coro}. \footnote{Corollary \ref{main coro} is based on the   algebro-geometric  fact (see e.g.  \cite[p. 13]{WCB2}) that  the orbit  $\mc O_{\rho}$ of an exceptional representation $\rho$  is Zariski open in a  certain affine space.}    The obtained conditions, which ensure maximality of the rank,  are as follows.

Let $\rho$, $\rho'$ be exceptional representations, $\alpha$, $\alpha'$ be their dimension vectors and let $A$, $A'$ be the supports of $\alpha$, $\alpha'$. 
When $Q$ has no edges loops and  $\alpha$ or $\alpha'$ has only one nontrivial value, i. e. $A$ or $A'$ is a single element set,   then $F_{\rho, \rho'}^Q$ has maximal rank (Lemma \ref{lemma with one simple representation}).

In Section \ref{without loops} we consider quivers without loops and exceptional representations whose dimension vectors are thin, i. e. the components of these vectors take values in $\{0,1\}$ (see Definition \ref{thin}).  The main result of this section (Lemma \ref{lemma with units}) is that, when the graph of $Q$ has no  loops,  for any two thin exceptional representations  $\rho$, $\rho'$  the linear map $F_{\rho,\rho'}^Q$  has maximal rank.  The last Lemma \ref{about A cap A' single element} of this section  considers some  cases in which $A\cap A'$ is a single element set
 and $\rho$, $\rho'$ are not restricted to be with thin  dimension vectors.

In Section \ref{star shaped} we restrict $Q$ further. We consider star shaped quivers with any orientation of the arrows (see Figure \ref{star shaped quiver}).  We allow here the exceptional representations to have hill dimension vector (Definition \ref{hill}) in addition to thin dimension vectors. It is  shown that  for any two exceptional representations $\rho, \rho' \in Rep_k(Q)$, s.t. one of them is thin and the dimension vector of the other is  hill or thin,  the map $F_{\rho,\rho'}^Q$ has maximal rank (Proposition  \ref{main coro sect 5}). It follows that in a star quiver with three arms  $Q$  for any two exceptional representations $\rho, \rho'$ we  have  ${\rm Hom}(\rho,\rho')=0$ or ${\rm Ext}^1(\rho,\rho')=0$ provided that $\rho$ or $\rho'$ is a thin representation, since in this case from \cite{Ringelv3} it follows that all exceptional representations are  thin or hill.

	 Star-shaped quivers   have been extensively  studied (going back to \cite{GP}, \cite{BGP} and recently e.g. \cite{AH}), but  to the best of our knowledge  Proposition  \ref{main coro sect 5} is new. 
	
	 Examples where   $\hom(\rho,\rho')\neq 0$ and  $\hom^1(\rho,\rho') \neq 0$ with exceptional representations $\rho,\rho'$ both having  hill dimension vectors in quiver with graph the extended Dynkin diagram $\wt{\mathbb E}_6$ were pointed to us by Claus Michael  Ringel.  Using \cite{DR} one can 
	 determine all the pairs $(\rho,\rho')$ of exceptional representations  such that $\hom(\rho,\rho')$ and  $\hom^1(\rho,\rho')$ are  both non-zero for extended Dynkin  quivers. 
			
From here till the end of the introduction $Q$ is a  Dynkin quiver $Q$ (i. e. the graph of $Q$ is $A_n$ with $n\geq 1$ or $D_n$ with $n\geq 4$  or $E_n$ with $n=6,7,8$).	Lemma \ref{lemma for leq 0 and ENC} combined with the positive answer of the question in Section \ref{question},	 and the positivity of the Euler form  imply  that there are no Ext-nontrivial couples in $Rep_k(Q)$ (Corollary \ref{noENCdynkin}). This in turn implies that $Rep_k(Q)$ is regularity preserving. Furthermore,  there are no $\sigma$-irregular objects for any $\sigma \in \st(D^b(Rep_k(Q)))$. 

Corollary \ref{noENCdynkin} means that for any two exceptional representations $\rho,\rho' \in Rep_k(Q)$ we have ${\rm Ext}^1(\rho, \rho')=0$ or ${\rm Ext}^1(\rho',\rho)=0$. Analogous property in  degree zero, which says that $\Hom(\rho,\rho')=0$ or $\Hom(\rho',\rho)=0$ for any two  non-equivalent  exceptional representations $\rho,\rho' \in Rep_k(Q)$, is well known for Dynkin quivers.  
 These  facts  about any  Dynkin quiver $Q$  can be summarized   by saying that  for any two   non-equivalent   exceptional representations  $\rho,\rho' \in Rep_k(Q)$  the product of  the two  numbers in each row and in each column  of   the  table below vanishes (see \textbf{\textit{Some notations}} for the notations $\hom(\rho,\rho'), \hom^1(\rho,\rho')$):
\begin{gather} \begin{array}{| c | c |}
  \hline 
	\hom(\rho,\rho') & \hom^1(\rho,\rho') \\ 
	  \hline 
	\hom(\rho',\rho) & \hom^1(\rho',\rho) \\ 
	  \hline 	
	\end{array}. \nonumber
\end{gather}

In the final Section \ref{directions for future} we  pose  some conjectures and questions.

\vspace{3mm}

\vspace{3mm}

 \textit{\textbf{Some notations.}} In these notes $k$ is an algebraically closed field. The letter ${\mathcal T}$  denotes   a triangulated category, linear over  $k$, the shift functor  in ${\mathcal T}$ is designated by $[1]$.   We write $\Hom^i(X,Y)$ for  $\Hom(X,Y[i])$ and  $\hom^i(X,Y)$ for  $\dim_k(\Hom(X,Y[i]))$, where $X,Y\in \mc T$. 

  An \textit{exceptional object}  is an object $E\in \mc T$ satisfying $\Hom^i(E,E)=0$ for $i\neq 0$ and  $\Hom(E,E)=k $. We denote by ${\mc A}_{exc}$, resp. $D^b(\mc A)_{exc}$,  the set of all
    exceptional objects of  $\mc A$, resp. of  $D^b(\mc A)$.

An \textit{exceptional collection} is a sequence $\mc E = (E_1,E_2,\dots,E_n)\subset \mc T_{exc}$ satisfying $\hom^*(E_i,E_j)=0$ for $i>j$.    If  in addition we have $\langle \mc E \rangle = \mc T$, then $\mc E$ will be called a full exceptional collection.

 An abelian category $\mc A$ is said to be hereditary, if ${\rm Ext}^i(X,Y)=0$ for any  $X,Y \in \mc A$ and $i\geq 2$,  it is said to be of finite length, if it is Artinian and Noterian.

\vspace{3mm}

\textit{{\bf Acknowledgements:}}
We are deeply grateful to Claus Michael  Ringel for his invaluable help by pointing to us a counterexample of a false statement in the previous version, Corollary 6.3. in 1405.2943v2.  

The authors wish to express their gratitude to    Maxim Kontsevich  and Tony Pantev   for their   interest in this paper. 
The authors wish to express also  their gratitude to   Pranav Pandit   for his  interest in the paper, and for pointing us the important  references \cite{Ringel}, \cite{BGP}, \cite{DR} and for other important comments.

The authors were funded by NSF DMS 0854977 FRG, NSF DMS 0600800, NSF DMS 0652633
FRG, NSF DMS 0854977, NSF DMS 0901330, FWF P 24572 N25, by FWF P20778 and by an
ERC Grant.

\section{The two examples  in \texorpdfstring{\cite{DK}}{\space} of  regularity preserving  categories with Ext-nontrivial couples} \label{two examples}
Here we recall some results and definitions of \cite[Section 2]{DK}, which are related to the content  of  the present paper. 
First we recall two  definitions

\begin{df} \label{Ext-nontrivial couple} An \uline{Ext-nontrivial couple} in a hereditary abelian category $\mc A$  is a couple of exceptional objects $\{ L,\Gamma \} \subset {\mc A}_{exc}$, s. t.
$\hom^1(L,\Gamma)\neq 0$ and $\hom^1(\Gamma,L)\neq 0$.

\end{df}

\begin{df} \label{def of ENC}
Let $\mc A$ be a hereditary category. We say that $\mc A$ has

   \textbf{RP Property 1}: if  for    each    Ext-nontrivial couple  $\{ \Gamma, \Gamma' \}\subset \mc A$  and    for each   $X \in {\mc A}_{exc}$ \\   from $\hom^*(\Gamma,X)=0 $ it follows $\hom^*(X,\Gamma')= 0$;

  \textbf{RP Property 2}: if for    each    Ext-nontrivial couple  $\{ \Gamma, \Gamma' \}\subset \mc A$  and for any two  $X,Y\in {\mc A}_{exc}$ \\
from  $ \hom(\Gamma,X)\neq 0, \hom(X,Y)\neq 0, \hom^*(\Gamma,Y) =
0$ it follows $ \hom(\Gamma',Y)\neq 0$.

\end{df}
 For any finite quiver $Q$ we denote   the category of $k$-representations
of $Q$ by $Rep_{k}(Q)$. This   category is a $\hom$-finite hereditary $k$-linear abelian (see e. g. \cite{WCB2}).
 In this section are discussed  the exceptional objects  and their pairwise relations  in $Rep_k(Q_1)$, $Rep_k(Q_2)$,  where:
\be \label{two quivers} \label{Q1} Q_1= \begin{diagram}[1em]
   &       &  \circ  &       &    \\
   & \ruTo &    & \luTo &       \\
\circ  & \rTo  &    &       &  \circ
\end{diagram}  \ \qquad \qquad \ Q_2= \begin{diagram}[1em]
 \circ &  \rTo  &  \circ    \\
  \uTo &        & \uTo     \\
\circ   & \rTo  &    \circ
\end{diagram}. \ee

Recall (see page 8 in \cite{WCB2}) that for  any quiver $Q$ and any  $\rho,\rho'\in Rep_k(Q)$ we have the formula
\begin{gather}\label{euler} \hom(\rho,\rho')-\hom^1(\rho,\rho')=\scal{\ul{\dim}(\rho),\ul{\dim}(\rho')},  \end{gather}
where  $\scal{,}$ is the Euler form of $Q$. In particular, it follows  that if $\rho \in Rep_k(Q)$ is an exceptional object, then $\scal{\ul{\dim}(\rho),\ul{\dim}(\rho)}=1$. The vectors satisfying this equality are called real roots(see \cite[p. 17]{WCB2}).   The real roots of $Q_1$ are $(m+1,m,m)$,$(m,m+1,m+1)$, $(m,m,m+1)$, $(m
 +1,m+1,m)$, $(m+1,m,m+1)$, $(m,m+1,m)$, $m\geq 0$. The imaginary roots\footnote{Imaginary root is a  vector $\rho$ with $\scal{\ul{dim}(\rho),\ul{dim}(\rho)}\leq 0$.}  of $Q_1$,  are $(m,m,m)$, $m\geq 1$. Not every real root is a dimension vector of an exceptional representation(see \cite[Lemma 2.1]{DK}).

Propositions \ref{exceptional objects in Q1} and
\ref{exceptional objects in Q2}   classify  the exceptional objects on $Rep_k(Q_1)$,
$Rep_k(Q_2)$. 
\begin{prop}[\cite{DK}] \label{exceptional objects in Q1} The exceptional objects up to isomorphism in  $ Rep_{k}(Q_1) $ are ($m=0,1,2,\dots$)
\begin{gather} E_1^m = \begin{diagram}[1em]
   &       &  k^m &       &    \\
   & \ruTo^{\pi_+^m} &    & \luTo^{Id} &       \\
k^{m+1}  & \rTo^{\pi_-^m}  &    &       &  k^m
\end{diagram} \ \ \ \  E_2^m = \begin{diagram}[1em]
   &       &  k^{m+1} &       &    \\
   & \ruTo^{j_+^m} &    & \luTo^{Id} &       \\
k^{m}  & \rTo^{j_-^m}  &    &       &  k^{m+1}
\end{diagram} \ \ \ \  E_3^m = \begin{diagram}[1em]
   &       &  k^{m+1} &       &    \\
   & \ruTo^{j_+^m} &    & \luTo^{j_-^m} &       \\
k^{m}  & \rTo^{Id}  &    &       &  k^{m}
\end{diagram} \nonumber  \\
E_4^m = \begin{diagram}[1em]
   &       &  k^m &       &    \\
   & \ruTo^{\pi_+^m} &    & \luTo^{\pi_-^m} &       \\
k^{m+1}  & \rTo^{Id}  &    &       &  k^{m+1}
\end{diagram} \ \ \ \ M = \begin{diagram}[1em]
   &       &  0 &       &    \\
   & \ruTo &    & \luTo &       \\
0  & \rTo  &    &       &  k
\end{diagram} \ \ \ \  M'= \begin{diagram}[1em]
   &       &  k &       &    \\
   & \ruTo^{Id} &   &  \luTo  &       \\
k  &  \rTo &    &       &  0
\end{diagram},\nonumber \end{gather}
where 
\begin{gather} 
\nonumber  \pi_+^m(a_1,a_2,\dots, a_m, a_{m+1}) =(a_1,a_2,\dots, a_m) \qquad   \pi_-^m(a_1,a_2,\dots, a_m, a_{m+1})=(a_2,\dots, a_m, a_{m+1}) \\
\nonumber  j_+^m(a_1,a_2,\dots, a_m) =(a_1,a_2,\dots, a_m,0)  \qquad
  j_-^m(a_1,a_2,\dots, a_m)=(0,a_1,\dots, a_m).
\end{gather}
\end{prop}

\begin{prop}[\cite{DK}] \label{exceptional objects in Q2} The exceptional objects up to isomorphism in  $ Rep_{k}(Q_2) $ are($m=0,1,2,\dots$)
\begin{gather} E_1^m = \begin{diagram}[1em]
 k^m &  \rTo^{Id}  &  k^m   \\
  \uTo^{\pi_+^m} &        & \uTo^{Id}     \\
k^{m+1}  & \rTo^{\pi_-^m}  &    k^m
\end{diagram} \ \ \ \  E_2^m = \begin{diagram}[1em]
 k^{m+1} &  \rTo^{Id}  &  k^{m+1}   \\
  \uTo^{j_+^m} &        & \uTo^{Id}     \\
k^{m}  & \rTo^{j_-^m}   &    k^{m+1}
\end{diagram} \ \ \ \  E_3^m = \begin{diagram}[1em]
 k^{m} &  \rTo^{j_+^m}  &  k^{m+1}   \\
  \uTo^{Id} &        & \uTo^{j_-^m}     \\
k^{m}  & \rTo^{Id}   &    k^{m}
\end{diagram}\ \ \ \  E_4^m = \begin{diagram}[1em]
 k^{m+1} &  \rTo^{\pi_+^m}  &  k^{m}   \\
  \uTo^{Id} &        & \uTo^{\pi_-^m}     \\
k^{m+1}  & \rTo^{Id}   &    k^{m+1}
\end{diagram}\nonumber  \\
E_5^m = \begin{diagram}[1em]
 k^m &  \rTo^{j_+^m}  &  k^{m+1}   \\
  \uTo^{Id} &        & \uTo^{Id}     \\
k^{m}  & \rTo^{j_-^m}  &    k^{m+1}
\end{diagram} \ \ \ \  E_6^m = \begin{diagram}[1em]
 k^{m+1} &  \rTo^{\pi_+^m}  &  k^{m}   \\
  \uTo^{Id} &        & \uTo^{Id}     \\
k^{m+1}  & \rTo^{\pi_-^m}   &    k^{m}
\end{diagram} \ \ \ \  E_7^m = \begin{diagram}[1em]
 k^{m} &  \rTo^{Id}  &  k^{m}   \\
  \uTo^{\pi_+^m} &        & \uTo^{\pi_-^m}     \\
k^{m+1}  & \rTo^{Id}   &    k^{m+1}
\end{diagram}\ \ \ \  E_8^m = \begin{diagram}[1em]
 k^{m+1} &  \rTo^{Id}  &  k^{m+1}   \\
  \uTo^{j_+^m} &        & \uTo^{j_-^m}     \\
k^{m}  & \rTo^{Id}   &    k^{m}
\end{diagram}\nonumber \\ F_+ = \begin{diagram}[1em]
 k &  \rTo  &  0   \\
  \uTo &        & \uTo     \\
0  & \rTo  &    0
\end{diagram} \ \ \ \  F_- = \begin{diagram}[1em]
 0 &  \rTo  &  0   \\
  \uTo &        & \uTo     \\
0  & \rTo   &    k
\end{diagram} \ \ \ \  G_+ = \begin{diagram}[1em]
 k &  \rTo^{Id}  &  k   \\
  \uTo^{Id} &        & \uTo    \\
k  & \rTo   &    0
\end{diagram}\ \ \ \  G_- = \begin{diagram}[1em]
 0 &  \rTo  &  k   \\
  \uTo &        & \uTo^{Id}     \\
k  & \rTo^{Id}   &    k
\end{diagram}.\nonumber \end{gather}
\end{prop}

In \cite[Subsection 2.2]{DK}  we  compute  $\hom(\rho,\rho')$, $\hom^1(\rho,\rho')$ with  $\rho,\rho'$ varying throughout the obtained lists in Propositions \ref{exceptional objects in Q1}, \ref{exceptional objects in Q2} and group these dimensions in  tables of series.
From these tables one finds that the only Ext-nontrivial  couple  in
$Rep_k(Q_1)$  is $\{M,M'\}$ and the
Ext-nontrivial couples in $Rep_k(Q_2)$ are  $\{F_+,G_-\}$,
$\{F_-,G_+\}$.

From the obtained tables with dimensions one verifies also the following properties. 
\begin{coro}[\cite{DK}] \label{RP property 1,2 and.. for Q1} The categories $Rep_{k}(Q_1)$, $Rep_{k}(Q_2)$ satisfy the following properties:
\begin{itemize}
    \item[\textbf{(a)}] RP property 1, RP property 2 (see Definition \ref{def of ENC}).
    \item[\textbf{(b)}] For any two exceptional objects $X, Y \in Rep_{k}(Q_i)$ at most one  degree in $\{ \hom^p(X,Y) \}_{p\in \ZZ}$ is nonzero, where  $i\in \{1,2\}$.
    \end{itemize}
\end{coro}

\section{  Is  \texorpdfstring{$\hom(\rho,\rho')=0$}{\space} or   \texorpdfstring{$\hom^1(\rho,\rho')=0$}{\space} for  exceptional representations   \texorpdfstring{$\rho,\rho'$}{\space} ?} \label{question}

From now on  $Q$ is  any connected quiver.   We denote the set of vertices  by $V(Q)$, the set of arrows by $Arr(Q)$,   and the underlying non-oriented graph by $\Gamma(Q)$. Let
\be \label{origin,end} Arr(Q) \rightarrow  V(Q)\times V(Q)  \ \  \ \ a \mapsto (s(a),t(a)) \in V(Q) \times V(Q) \ee  be the   function assigning
 to an arrow $a\in Arr(Q)$ its origin $s(a) \in V(Q)$ and its end $t(a) \in V(Q)$. We recall that    the Euler form $\langle , \rangle$ of $Q$ is:
\begin{gather}  \langle \alpha , \beta \rangle_Q = \sum_{i \in V(Q)}  \alpha_i  \beta_i - \sum_{a \in Arr(Q)}  \alpha_{s(a)}  \beta_{t(a)} \qquad  \alpha , \beta \in \ZZ^{V(Q)}. \end{gather}
 
The dual quiver $Q^{\vee}$ has $V(Q^{\vee})=V(Q)$, $Arr(Q^{\vee})=Arr(Q)$, but $(s^{\vee},t^{\vee})=(t,s)$. 
By  transposing matrices we obtain an equivalence \begin{gather} \label{vee} \bd Rep_k(Q)^{op} & \rTo^{\vee} &  Rep_k(Q^{\vee}). \ed \end{gather}
The following properties hold
\begin{gather} \label{vee1} \forall \rho, \rho' \in   Rep_k(Q)    \qquad \ul{\dim}(\rho)=\ul{\dim}(\rho^{\vee}) \quad \hom_{Q}^i(\rho,\rho')=\hom_{Q^{\vee}}^i(\rho'^{\vee},\rho^{\vee}) \\ 
 \label{vee2} \forall \alpha, \beta \in \NN^{V(Q)}      \qquad \langle \alpha, \beta \rangle_{Q^{\vee}}=\langle \beta, \alpha \rangle_{Q}  \end{gather}
The basic observation of this paper is:
\begin{lemma}\label{lemma for leq 0 and ENC} If  any two exceptional objects $\rho,\rho' \in Rep_k(Q)$ satisfy $\hom(\rho,\rho')=0$ or  $\hom^1(\rho,\rho')=0$, then any Ext-nontrivial couple  $\{\rho, \rho'\}$ satisfies  $\langle \ul{\dim}(\rho)+\ul{\dim}(\rho'), \ul{\dim}(\rho)+\ul{\dim}(\rho')\rangle \leq 0$. 
\end{lemma}
\bpr Since $\rho, \rho'$ are exceptional representations, we have $ \langle \ul{\dim}(\rho), \ul{\dim}(\rho)\rangle=\langle \ul{\dim}(\rho'), \ul{\dim}(\rho')\rangle=1$, therefore  \begin{gather} \label{lemma for leq 0 and ENC eq}  \langle \ul{\dim}(\rho)+\ul{\dim}(\rho'), \ul{\dim}(\rho)+\ul{\dim}(\rho')\rangle=  2 +  \langle \ul{\dim}(\rho), \ul{\dim}(\rho')\rangle+ \langle\ul{\dim}(\rho'), \ul{\dim}(\rho)\rangle.  \end{gather}
Since  $\hom^1(\rho, \rho')\neq 0$, $\hom^1(\rho', \rho)\neq 0$, by the given property of the exceptional objects we obtain  $\hom(\rho, \rho')=\hom(\rho', \rho)= 0$, hence by \eqref{euler}  we obtain $ \langle \ul{\dim}(\rho), \ul{\dim}(\rho')\rangle=-\hom^1(\rho, \rho')<0$, $ \langle \ul{\dim}(\rho'), \ul{\dim}(\rho)\rangle=-\hom^1(\rho', \rho)<0$. Now the lemma follows from  \eqref{lemma for leq 0 and ENC eq}.
\epr 
In Corollary \ref{RP property 1,2 and.. for Q1} (b) we see that the condition of the lemma above is satisfied in $Rep_k(Q_1)$, $Rep_k(Q_2)$. 
 This condition is satisfied in all Dynkin quivers as well. More precisely, in  \cite[p. 59]{Ringel} is shown that for any two exceptional representations $\rho, \rho'$ in $Rep_k(Q)$  for Dynkin $Q$ we have $\hom(\rho, \rho')=0$ or $\hom^1(\rho,\rho')=0$   ( here C. M. Rignel uses  the fact that Dynkin quivers are representation directed   \cite{BGP}  \cite{DR}).\footnote{We thank Pranav Pandit for pointing us this fact.}   Now it follows:

\begin{coro} \label{noENCdynkin} If $Q$ is  a Dynkin quiver, then there are no Ext-nontrivial  couples in $Rep_k(Q)$, i. e. for any two exceptional representations $\rho,\rho' \in Rep_k(Q)$ we have $\hom^1(\rho,\rho')=0$ or $\hom^1(\rho',\rho)=0$. 
\end{coro}
\bpr Recall that for such a quiver we have $\langle \alpha, \alpha \rangle >0$ for each $\alpha\in \NN^{V(Q)} \setminus \{0\}$. Since  for any two exceptional representations $\rho,\rho' \in Rep_k(Q)$ we have $\hom(\rho,\rho')=0$ or $\hom^1(\rho,\rho')=0$, we can apply   Lemma \ref{lemma for leq 0 and ENC}. The corollary follows.   \epr

\begin{coro} \label{Dynkin are regularity preserv} If $Q$ is  a Dynkin quiver, then  $Rep_k(Q)$ is regularity preserving category. Furthermore, there are no $\sigma$-irregular objects for any $\sigma \in \st(D^b(Rep_k(Q)))$.
\end{coro} 
\bpr  Since there are no Ext-nontrivial couples,  RP properties 1,2(Definition \ref{def of ENC}) are tautologically satisfied. Then by \cite[Proposition 6.6]{DK} $Rep_k(Q)$ is regularity preserving. Actually, due to \cite[Lemma 6.3]{DK}, there are no $\sigma$-irregular objects for any $\sigma \in \st(D^b(Rep_k(Q)))$.
\epr

\section{The differential  \texorpdfstring{$F_{\rho,\rho'}^Q$}{\space}} \label{the linear map}

The condition of Lemma \ref{lemma for leq 0 and ENC} is related to the   standard differential in the 2-term complex computing $\RR \Hom_Q(\rho,\rho')$. We recall this definition:
\begin{df} \label{the map F_rho rho'} For any two representations $\rho,\rho' \in Rep_k(Q)$ we  denote  a by  $F^Q_{\rho,\rho'}$ (we omit the superscript $Q$, when it is clear which is  the quiver  in question) the   standard differential in the 2-term complex computing $\RR \Hom_Q(\rho,\rho')$, which is defined as follows:  
\begin{gather}  F^Q_{\rho, \rho'} : \prod_{i\in V(Q)} \Hom(k^{\alpha_i},k^{\alpha'_i}) \rightarrow \prod_{a\in Arr(Q)} \Hom(k^{\alpha_{s(a)}}, k^{\alpha'_{t(a)}}) \end{gather}
  where
	$\alpha= \ul{\dim}(\rho),$ $\alpha'= \ul{\dim}(\rho')\in \NN^{V(Q)}$, and:
		\begin{gather}  F^Q_{\rho, \rho'}\left (\{f_i\}_{i\in V(Q)} \right ) = \{f_{t(a)}\circ \rho_a - \rho'_a \circ f_{s(a)} \}_{a \in Arr(Q)}. \end{gather}
\end{df}
This differential  will be used to obtain the condition in Lemma \ref{lemma for leq 0 and ENC}. More precisely, we have the following standard facts: 
\begin{lemma}  \label{lemma about F_rhi,rhi'} Let $Q$ be a quiver and $\rho, \rho' \in Rep_k(Q)$  be  two representations.  The following hold: 

{\rm (a)}    $\Hom_{Rep_k(Q)}(\rho,\rho')=\ker\left (F^Q_{\rho,\rho'} \right )$

{\rm (b) }  $\langle \ul{\dim}(\rho) , \ul{\dim}(\rho') \rangle =\dim\left ({\rm dom}\left ( F^Q_{\rho, \rho'} \right ) \right )-\dim\left ({\rm cod}\left ( F^Q_{\rho, \rho'} \right ) \right )$, where ${\rm dom}\left ( F^Q_{\rho, \rho'} \right )$ and \\ ${\rm cod}\left ( F^Q_{\rho, \rho'} \right ) $ denote the domain and codomain of $F_{\rho,\rho'}^Q$.

{\rm (c) }  Let  $\langle \ul{\dim}(\rho), \ul{\dim}(\rho')\rangle \geq 0$. Then  $F^Q_{\rho,\rho'}$ has maximal rank   iff  $\hom(\rho,\rho')=\langle \ul{\dim}(\rho), \ul{\dim}(\rho')\rangle$ and $\hom^1(\rho,\rho')=0$. 

{\rm (d) }  Let $\langle \ul{\dim}(\rho), \ul{\dim}(\rho')\rangle < 0$. Then $F^Q_{\rho,\rho'}$ has maximal rank iff $\hom(\rho,\rho')=0$ and $\hom^1(\rho,\rho')=-\langle \ul{\dim}(\rho), \ul{\dim}(\rho')\rangle$. 

{\rm (e) }  Let $Q^{\vee}$ be the dual quiver and $\vee$ be the equivalence in \eqref{vee}, then  $F^Q_{\rho,\rho'}$   has maximal rank iff  $F^{Q^\vee}_{\rho'^{\vee},\rho^{\vee}}$  has maximal rank 
\end{lemma}
\bpr (a) and (b) follow from the definitions, (c) and (d) follow from (a), (b) and \eqref{euler}. Finally, (e) follows from  (c), (d),  \eqref{vee1}, and \eqref{vee2}.   \epr

For any $\alpha \in \NN^{V(Q)}$ we denote \begin{gather} GL(\alpha) = \prod_{i\in V(Q)} GL(\alpha_i,k); Rep(\alpha)=\{\rho \in Rep_k(Q) : \ul{\dim}(\rho)=\alpha \}=\prod_{a\in Arr(Q)} \Hom(k^{\alpha_{s(a)}},k^{\alpha_{t(a)}}). \nonumber \end{gather}  For any $\alpha \in \NN^{V(Q)}$ the isomorphism classes of representations with dimension vector $\alpha$ are the orbits of the left action: 
\begin{gather}\label{left action} GL(\alpha) \times Rep(\alpha) \rightarrow  Rep(\alpha) \qquad g.\rho= \{g(t(a))\circ \rho_a \circ g(s(a))^{-1}\}_{a \in Arr(Q)}.  \end{gather}
For $\rho \in Rep(\alpha)$ the orbit containing $\rho$ is denoted by ${\mc O}_{\rho}$.

 Let $\alpha, \alpha'\in \NN^{V(Q)}$, $ g\in GL(\alpha) $,  $g'\in GL(\alpha')$. It is easy to show that for any $\rho \in Rep(\alpha)$, $\rho' \in Rep(\alpha')$ we have  
\begin{gather}\label{action of g} F_{g.\rho,\  \rho'}= R_{g^{-1}}\circ F_{\rho, \rho'} \circ  R_{g} \qquad  F_{\rho, \ g'.\rho'}= L_{g'}\circ F_{\rho, \rho'} \circ  L_{g'^{-1}},   \end{gather}  where:   
\begin{gather} L_{g'}, R_g: \prod_{i\in V(Q)} \Hom(k^{\alpha_i},k^{\alpha'_i})  \rightarrow  \prod_{i\in V(Q)} \Hom(k^{\alpha_i},k^{\alpha'_i})  \nonumber \\[-3mm] \\  L_{g'}\left (\{f_i\}_{i\in V(Q)} \right )= \{g'_i\circ f_i\}_{i\in V(Q)},  \ \  \ R_{g}\left (\{f_i\}_{i\in V(Q)} \right )= \{ f_i\circ g_i\}_{i\in V(Q)}; \nonumber \end{gather}

\begin{gather}   L_{g'}, R_{g}: \prod_{a\in Arr(Q)} \Hom(k^{\alpha_{s(a)}}, k^{\alpha'_{t(a)}}) \rightarrow  \prod_{a\in Arr(Q)} \Hom(k^{\alpha_{s(a)}}, k^{\alpha'_{t(a)}})  \nonumber \\[-3mm]  \\  L_{g'}\left (\{u_a\}_{a\in Arr(Q)} \right )= \{g'_{t(a)}\circ u_a\}_{a\in Arr(Q)} \ \  R_{g}\left (\{u_a\}_{a\in Arr(Q)} \right )= \{u_a \circ g_{s(a)} \}_{a\in Arr(Q)}.\nonumber \end{gather}

In particular, we see immediately  that 
\begin{lemma} Let  $\alpha, \alpha'\in \NN^{V(Q)}$, $(\rho, \rho') \in Rep(\alpha)\times  Rep(\alpha')$. If $F_{\rho, \rho'}$ is not of maximal rank, then $F_{x,y}$ is not of maximal rank for any $(x,y) \in \mc O_{\rho} \times \mc O_{\rho'}$.
\end{lemma}
The following corollary will be widely used.
\begin{coro} \label{main coro} Let $\alpha, \alpha' \in \NN^{V(Q)}$ be real roots of $Q$. Let $\rho \in Rep(\alpha)$, $\rho' \in Rep(\alpha')$  be exceptional representations. 
 If $F_{x,y}$ has maximal rank for some $(x,y)\in  Rep(\alpha)\times Rep(\alpha')$, then $F_{\rho,y}$ and  $F_{x,\rho'}$  have maximal rank. 
For each $a \in Arr(Q)$ the linear maps $\rho_a$, $\rho'_a$ have maximal rank.
\end{coro}
\bpr First recall that,  since $\rho,\rho'$ are exceptional,  the orbits  $\mc O_{\rho}$ and   $\mc O_{\rho'}$ are Zariski open in $Rep(\alpha)$ and $Rep(\alpha')$, respectively (see \cite[p. 13]{WCB2}). For  a  given $x \in  Rep(\alpha)$ the condition on $y\in Rep(\alpha')$ to be such that $F_{x,y}$ is not of maximal rank is expressed by vanishing of certain family of polynomials on $Rep(\alpha')$.    If there is $y\in Rep(\alpha')$ such that $F_{x,y}$ is  of maximal rank, then the  zero set of this family of polynomials is a proper Zariski closed subset of $Rep(\alpha')$, hence, by the previous lemma, not maximality of the rank of  $F_{x,\rho'}$ implies that the orbit $\mc O_{\rho'}$ is contained in this proper zariski closed subset, and then $\mc O_{\rho'}$ can not be an open subset of $Rep(\alpha')$. Thus, we showed that if $F_{x,y}$ is  of maximal rank  for some $y\in Rep(\alpha')$, then  $F_{x,\rho'}$ is of maximal rank. The claim about $F_{\rho,y}$   is proved by the same arguments applied to $\rho$. 

 Finally, the property that $\rho_a$ is not of maximal rank is invariant under the action of $GL(\alpha)$,  for any $a\in Arr(Q)$. It follows that  non-maximality of the rank of $\rho_a$ implies that $\mc O_{\rho}$ is contained in a proper Zariski closed subset of $Rep(\alpha)$. If $\rho$ is an exceptional representation, then  $\mc O_{\rho}$  is Zariski open in  $Rep(\alpha)$, therefore  $\rho_a$  is of maximal rank for each $a\in Arr(Q)$.
\epr

It is useful to  give a more precise description of  the map defined in Definition \ref{the map F_rho rho'}. For any $(\rho,\rho') \in Rep(\alpha)\times Rep(\alpha')$ we denote $A= \{i\in V(Q) : \alpha_i \neq 0\}$,  $A'= \{i\in V(Q) : \alpha'_i \neq 0\}$. We denote also $Arr(A,A')=\{a \in Arr(Q): s(a)\in A, t(a)\in A'\}$. Then for $F_{\rho,\rho'}$ we can write 

\begin{gather} \label{F_rho,rho'1} F_{\rho, \rho'} : \prod_{i\in A\cap A'} \Hom(k^{\alpha_i},k^{\alpha'_i}) \rightarrow \prod_{a\in Arr(A,A')} \Hom(k^{\alpha_{s(a)}}, k^{\alpha'_{t(a)}})  \\ 
\label{F_rho,rho'2} F_{\rho, \rho'}\left (\{f_i\}_{i\in A\cap A'} \right ) =\left \{  \begin{array}{cc}
f_{t(a)}\circ \rho_a - \rho'_a \circ f_{s(a)}  & a \in Arr(A\cap A',A\cap A')  \\
 - \rho'_a \circ f_{s(a)}  & a \in Arr(A\cap A',A' \setminus A)  \\
f_{t(a)}\circ \rho_a & a \in Arr(A \setminus A',A\cap A') \\
0 & a \in Arr(A \setminus A',A'\setminus A) \end{array}  \right \}. \end{gather}
In the rest sections we will study the question about maximality of the rank of $F_{\rho,\rho'}$, where $\rho$, $\rho'$ are exceptional representations. We prove first the following lemma:

\begin{lemma} \label{lemma with one simple representation} Let $Q$ have no edges loops.  Let $\rho,\rho'\in Rep_k(Q)$ be exceptional representations. If $A$ or $A'$ is a single element set, then $F_{\rho,\rho'}$ has maximal rank.
\end{lemma}
\begin{proof} Recall that we denote $\alpha = \ul{\dim}(\rho)$, $\alpha' = \ul{\dim}(\rho')$. 

 Assume that $A=\{j\}$. If $A \cap A' = \emptyset$, then $F_{\rho,\rho'}$ is injective. Let $A\cap A' =\{j\}$. 

Now $A\setminus A' =\emptyset$ and, since  $Q$ has no edges loops, we have $Arr(A\cap A', A\cap A')=\emptyset$,  hence for any $y\in Rep(\alpha')$ the map $F_{\rho,y}$ has the form (we use \eqref{F_rho,rho'1}, \eqref{F_rho,rho'2} and that now $\alpha_j=1$):
\begin{gather}  F_{\rho, y} :  \Hom(k,k^{\alpha'_j}) \rightarrow \prod_{a\in  Arr(\{j\},A' \setminus \{j\})  } \Hom(k, k^{\alpha'_{t(a)}}) \nonumber  \\ 
F_{\rho, y}\left (f\right ) =\{ - y_a \circ f  \}_{ a \in Arr(\{j\},A' \setminus \{j\}) }\nonumber  \end{gather} 
From this description of $  F_{\rho, y}$ we see that  we can choose  $y\in Rep(\alpha')$  so that  $F_{\rho, y}$ has maximal rank. Therefore by Corollary \ref{main coro}    $F_{\rho, \rho'}$ has maximal rank as well.  

If $A'=\{j\}$, then by the already proved $F^{Q^{\vee}}_{\rho'^{\vee}, \rho^{\vee}}$ has a maximal rank. Now we apply Lemma \ref{lemma about F_rhi,rhi'}, (e) and obtain that  $F_{\rho, \rho'}$  has maximal rank.
\end{proof}
An example of a quiver $Q$ and exceptional representations $\rho$, $\rho'$, s. t. $F^{Q}_{\rho,\rho'}$ is not of maximal rank is as follows: \begin{gather} \label{non maximal rank} Q=\begin{diagram}[1.5em]
 \circ &  \rTo  &  \circ    \\
  \uTo &        & \dTo     \\
\circ   & \rTo  &    \circ
\end{diagram} \quad \rho= \begin{diagram}[1.5em]
 k &  \rTo  &  0    \\
  \uTo &        & \dTo     \\
k   & \rTo  &    k
\end{diagram} \quad \rho'= \begin{diagram}[1.5em]
 0 &  \rTo  &  k    \\
  \uTo &        & \dTo     \\
k   & \rTo  &    k
\end{diagram}.    \end{gather}

One easily computes $\hom(\rho,\rho')=1$, $\langle \ul{\dim}(\rho), \ul{\dim}(\rho') \rangle=0$, and hence $\hom^1(\rho,\rho')=1$. Furthermore, $\rho$, $\rho'$ are exceptional representations.   Now from Lemma \ref{lemma about F_rhi,rhi'} (c) it follows that $F_{\rho,\rho'}$ is not of maximal rank.  Comparing with Corollary \ref{RP property 1,2 and.. for Q1} (b) we obtain 
\begin{lemma} \label{not equivalent} The categories of representations and their bounded derived categories  of the quivers  $ \begin{diagram}[1em]
 \circ &  \rTo  &  \circ    \\
  \uTo &        & \dTo     \\
\circ   & \rTo  &    \circ
\end{diagram}, \ \  \begin{diagram}[1em]
 \circ &  \rTo  &  \circ    \\
  \uTo &        & \uTo     \\
\circ   & \rTo  &    \circ
\end{diagram}  $ are not equivalent. 
\end{lemma}

In the next section we restrict our considerations to a quiver $Q$ without loops. 

\section{Remarks about  \texorpdfstring{$F_{\rho,\rho'}$}{\space} in  quivers without loops} \label{without loops}
Throughout     this section $Q$ is quiver without loops (i. e. the underlying graph $\Gamma(Q)$ is simply connected),  in particular  there is at most one edge between any two vertices of $Q$.
Here  we consider exceptional representations whose dimension vectors take values in $\{0,1\}$. These  exceptional representations are said to have thin dimension vector (Definition \ref{thin}). The main result of this section is that, when the graph of $Q$ has no loops, then for any two exceptional representations  $\rho$, $\rho'$  with thin dimension vectors the linear map $F_{\rho,\rho'}^Q$  has maximal rank. The last Lemma \ref{about A cap A' single element}  of this section considers some  cases in which $A\cap A'$ is a single element set
 and where $\rho$, $\rho'$ are not restricted to be with thin dimension vectors.

 For any subset $X\subset V(Q)$ we denote by $Q_X$ the quiver with $V(Q_X)=X$ and $Arr(Q_X)=Arr(X,X)$.
We denote by $\rho$, $\rho'$ two exceptional  representations of $Q$. We denote by $\alpha, \alpha' \in \NN^{V(Q)}$ their dimension vectors, and by $A=supp(\alpha)\subset V(Q)$, $A'=supp(\alpha')\subset V(Q)$ the supports of $\alpha, \alpha'$.  If $Arr(A\setminus A', A'\setminus A) \not = \emptyset$, then by the simply-connectivity of $Q$ it follows that $A\cap A' = \emptyset$ and then  $F_{\rho,\rho'}$ is  trivially injective(see \eqref{F_rho,rho'1}).\footnote{Note that since $\rho, \rho'$ are indecomposable, it follows that  $Q_A$,  $Q_{A'}$ are connected.} Thus, we see
\begin{lemma} \label{lemma about A setminus A'} If $Arr(A\setminus A', A'\setminus A) \not = \emptyset$, then $F_{\rho,\rho'}$ has  maximal rank. \end{lemma}
\textbf{From now on we assume that  $Arr(A\setminus A', A'\setminus A)  = \emptyset$,  and then the last row in \eqref{F_rho,rho'2} can be erased, and we have a disjoint union:}
\begin{gather} \label{Arr(A,A')} Arr(A,A')=  Arr(A\cap A',A\cap A')  \cup Arr(A\cap A',A' \setminus A) \cup Arr(A \setminus A',A\cap A'). \end{gather}

Now we  consider exceptional representations $\rho,\rho'$ whose dimension vectors contain only units and zeroes.   More precisely:
\begin{df} \label{thin} A vector $\alpha \in \NN^{V(Q)}$ is said to be thin if  for any $i\in A$ we have   $ \alpha_i=1$, where  $A=supp(\alpha)\subset V(Q)$ is the support of $\alpha$. 
\end{df}
\begin{remark} \label{thin remark} If  $\rho \in Rep_k(Q)$ is an exceptional representation with  a thin dimension vector (thin exceptional representation),  then the sub-quiver  $Q_A$ mist be connected  and  one can assume that  $\forall a \in Arr(A,A)\ \ \ \rho_a ={\rm Id}_k$. 
\end{remark}
\begin{lemma} \label{lemma with units} Let $\rho$ and $\rho'$ be exceptional representations with thin  dimension vectors.	
	Then $F_{\rho,\rho'}$ has maximal rank.
\end{lemma}
\bpr Due to the given conditions we can write:
\begin{gather} \label{F_rho,rho'10} \langle \ul{\dim}(\rho), \ul{\dim}(\rho')\rangle = \# (A\cap A')- \#(Arr(A,A')) \\ 
\label{F_rho,rho'11} \bd \prod_{i\in A\cap A'} k & \rTo^{ F_{\rho, \rho'}} & \prod_{a\in Arr(A,A')} k \ed \qquad  F_{\rho, \rho'}\left (\{f_i\}_{i\in A\cap A'} \right ) =\left \{  \begin{array}{cc}
f_{t(a)}-  f_{s(a)}  & a \in Arr(A\cap A',A\cap A')  \\
 -  f_{s(a)}  & a \in Arr(A\cap A',A' \setminus A)  \\
f_{t(a)} & a \in Arr(A \setminus A',A\cap A') \end{array}  \right \}. \nonumber \end{gather}
We can assume that $A\cap A' \not = \emptyset$. Since $Q_A$ and $Q_{A'}$ are connected, it follows that $Q_{A\cap A'}=Q_A \cap Q_{A'}$ is connected. Since there are no loops in $\Gamma(Q)$,  the graph of $Q_{A\cap A'}$ is simply connected, therefore 
$  \#(A\cap A')= \#(Arr(A\cap A', A\cap A')) + 1$. Putting \eqref{Arr(A,A')} and the latter equality  in \eqref{F_rho,rho'10}  we obtain
\begin{gather} \label{euler in the case of units} \langle \ul{\dim}(\rho), \ul{\dim}(\rho')\rangle = 1-\#( Arr(A\cap A', A'\setminus A))-\#( Arr(A\setminus A', A\cap A')). \end{gather}
The following lemma will be helpful for the rest of the proof
\begin{lemma} \label{simple lemma} Let $T$ be a  quiver, s. t. $\Gamma(T)$ is simply-connected.   Consider the linear map 
\begin{gather} \label{tree map1} F: \prod_{i\in V(T)} k \rightarrow \prod_{a\in Arr(T)} k \qquad  \qquad  
 F\left (\{f_i\}_{i\in V(T)} \right ) =\left \{  
f_{t(a)}-  f_{s(a)} \right \}_{  a \in Arr(T) } \end{gather} 
For each $j\in V(T)$, each $x\in k$, and each $y\in \prod_{a\in Arr(T)} k$ there exists unique $\{f_i\}_{i\in V(T)} \in  \prod_{i\in V(T)} k $ with $f_j=x$ and  $F(\{f_i\}_{i\in V(T)} )=y$. In particular, for each  $j\in V(T)$ the linear map 
\begin{gather} \label{tree map2}  \prod_{i\in V(T)} k \rightarrow k \oplus \prod_{a\in Arr(T)} k \qquad  \qquad  
 \{f_i\}_{i\in V(T)}  \mapsto \left (f_j,  \left \{  
f_{t(a)}-  f_{s(a)} \right \}_{  a \in Arr(T) } \right ) \end{gather} 
is isomorphism. 
\end{lemma}
\bpr  Easy induction on the number of vertices.  \epr
We will apply this lemma to  $Q_{A\cap A'}$.

Consider first the case $\langle \ul{\dim}(\rho), \ul{\dim}(\rho')\rangle \geq 0$. We need to show that $F_{\rho,\rho'}$ is surjective. Now by \eqref{euler in the case of units}  we have $1\geq Arr(A\cap A', A'\setminus A)+ Arr(A\setminus A', A\cap A')$,  and then the map $F_{\rho,\rho'}$ is the same as one of the maps \eqref{tree map1} or \eqref{tree map2} corresponding to $T=Q_{A\cap A'}$, hence $F_{\rho,\rho'}$  is surjective.

In the case $\langle \ul{\dim}(\rho), \ul{\dim}(\rho')\rangle < 0$,  we have $1< Arr(A\cap A', A'\setminus A)+ Arr(A\setminus A', A\cap A')$. Hence, for some projection $\pi$  the map $\pi \circ F_{\rho,\rho'}$ is the same as   the map \eqref{tree map2} corresponding to $T=Q_{A\cap A'}$, hence $F_{\rho,\rho'}$  is injective.
\epr
In the end we consider the map $F_{\rho,\rho'}$ in the case, when $A\cap A' $ has a single element.
\begin{lemma}  \label{about A cap A' single element} Let $\rho, \rho'$ be exceptional representations, s. t. $A\cap A' =\{j\}$. Let $\Gamma(Q)$ does not split at $j$, i. e. the edges adjacent to $j$ can be represented as follows $\bd[1em]  x & \lLine & j & \rLine & y \ed$ . Finally, assume that  $\alpha$ is constant  on $A\cap \{x,y,j\}$   or $\alpha'$ is constant on $A'\cap \{x,y,j\}$. Then $F_{\rho,\rho'}$ has maximal rank.
\end{lemma}
\bpr Since there are no loops in $\Gamma(Q)$, we have $Arr(A\cap A', A\cap A')=\emptyset$ and  $F_{\rho, \rho'} $ has the form
\begin{gather} \label{langle rangle} \langle \ul{\dim}(\rho),\ul{\dim}(\rho') \rangle = \alpha_j \alpha'_j- \sum_{a\in Arr(\{j\},A' \setminus A)} \alpha_j \alpha'_{t(a)} -  \sum_{a\in Arr(A \setminus A',\{j\})} \alpha_{s(a)} \alpha'_{j}\\   \label{F_rho,rho'1xx} F_{\rho, \rho'} :  \Hom(k^{\alpha_j},k^{\alpha'_j}) \rightarrow \prod_{a\in Arr(\{j\},A' \setminus A) } \Hom(k^{\alpha_{j}}, k^{\alpha'_{t(a)}}) \oplus\prod_{a\in Arr(A \setminus A',\{j\}) } \Hom(k^{\alpha_{s(a)}}, k^{\alpha'_{j}})  \\ 
\label{F_rho,rho'2xx} F_{\rho, \rho'}\left (\{f\} \right ) =\left \{  \begin{array}{cc}
 - \rho'_a \circ f  & a \in Arr(\{j\},A' \setminus A)  \\
f \circ \rho_a & a \in Arr(A \setminus A',\{j\})  \end{array}  \right \}. \end{gather}
From Lemma \ref{lemma with one simple representation} we can assume that $\#(A)\geq 2$,  $\#(A')\geq 2$. Since $\rho, \rho'$ are exceptional representations, $Q_A$ and $Q_{A'}$ are connected. Then the  edges adjacent to $j$ can be represented as follows $\bd A\setminus A' \ni x & \lLine & j & \rLine & y \in A'\setminus A \ed $ . We consider three cases.

If $Arr(\{j\},A' \setminus A)\neq \emptyset$,  $Arr(A \setminus A',\{j\})= \emptyset$, then we can represent the  arrows adjacent to $j$ as follows 
\be \bd A\setminus A' \ni x & \lTo & j & \rTo^a & y \in A'\setminus A \ed \ee and  $ \langle \ul{\dim}(\rho),\ul{\dim}(\rho') \rangle = \alpha_j \alpha'_j-\alpha_j \alpha'_{y}=\alpha_j (\alpha'_j- \alpha'_{y}) $, $F_{\rho, \rho'}\left (\{f\} \right )=-\rho'_a \circ f$. Since $\rho'$ is an exceptional representation, the map $\rho'_a$ has maximal rank (see the last part of Corollary \ref{main coro} ).  Therefore $F_{\rho, \rho'}$ has maximal rank. 

If $Arr(\{j\},A' \setminus A)= \emptyset$,  $Arr(A \setminus A',\{j\})\neq  \emptyset$, then we can represent the  arrows adjacent to $j$ as follows 
\be \bd A\setminus A' \ni x & \rTo^a & j & \lTo & y \in A'\setminus A \ed \ee and  $ \langle \ul{\dim}(\rho),\ul{\dim}(\rho') \rangle = \alpha_j \alpha'_j-\alpha_x \alpha'_{j}=(\alpha_j -\alpha_x) \alpha'_j $, $F_{\rho, \rho'}\left (\{f\} \right )=f \circ \rho_a $. Since $\rho$ is an exceptional representation, then $\rho_a$ has maximal rank.  Therefore $F_{\rho, \rho'}$ has maximal rank. 

If $Arr(\{j\},A' \setminus A)\neq \emptyset$,  $Arr(A \setminus A',\{j\})\neq  \emptyset$, then we can represent the  arrows adjacent to $j$ as follows 
\be \bd A\setminus A' \ni x & \rTo^a & j & \rTo^b & y \in A'\setminus A \ed \ee and  $ \langle \ul{\dim}(\rho),\ul{\dim}(\rho') \rangle = \alpha_j \alpha'_j-\alpha_x \alpha'_{j}-\alpha_j \alpha'_{y}$, $F_{\rho, \rho'}\left (\{f\} \right )=(-\rho'_b\circ f,f \circ \rho_a )$. 

If $\alpha$ is constant on $A\cap \{x,y,j\}$, then $\alpha_x=\alpha_j$ and $\langle \ul{\dim}(\rho),\ul{\dim}(\rho') \rangle=-\alpha_j \alpha'_{y} <0$. Since $\alpha_x=\alpha_j$  and $\rho_a$ has maximal rank, it follows that $\rho_a$ is isomorphism, hence $F_{\rho, \rho'}$  is injective.  Therefore $F_{\rho, \rho'}$ has maximal rank.

If $\alpha'$ is constant on $A'\cap \{x,y,j\}$, then $\alpha'_y=\alpha'_j$ and $\langle \ul{\dim}(\rho),\ul{\dim}(\rho') \rangle=-\alpha'_j \alpha_{x} <0$. Since  $\alpha'_y=\alpha'_j$  and $\rho'_b$ has maximal rank, it follows that $\rho'_b$ is isomorphism, hence $F_{\rho, \rho'}$  is injective.  Therefore $F_{\rho, \rho'}$ has maximal rank. The lemma is completely proved.
\epr

\section{Remarks about \texorpdfstring{$F_{\rho,\rho'}$}{\space} in  star shaped quivers} \label{star shaped}
In this section \ref{star shaped} we restrict $Q$ further. We assume that its graph is of the  type($m,n,p\geq2$):

\be \label{star shaped quiver} \Gamma(Q) = \begin{diagram}[size=1em]
    &        &     &        &        &        &     &         &           &           &       &           &  v_1  \\
    &        &     &        &        &        &     &         &           &           &       &   \ruLine &      \\
    &        &     &        &        &        &     &         &           &           &  v_2  &           &      \\
    &        &     &        &        &        &     &         &           & \iddots   &       &           &      \\
	  &        &     &        &        &        &     &         &  v_{n-1}  &           &       &           &      \\
    &        &     &        &        &        &     & \ruLine &           &           &       &           &      \\
u_1 & \rLine & u_2 & \cdots & u_{m-1}& \rLine & s   &         &           &           &       &           &      \\
    &        &     &        &        &        &     & \rdLine &           &           &       &           &      \\
    &        &     &        &        &        &     &         &  w_{p-1}  &           &       &           &      \\		
    &        &     &        &        &        &     &         &           &  \ddots   &       &           &      \\		
    &        &     &        &        &        &     &         &           &           &   w_2 &           &      \\	
    &        &     &        &        &        &     &         &           &           &       &  \rdLine  &      \\	
    &        &     &        &        &        &     &         &           &           &       &           &  w_1								
\end{diagram}
\ee
 For  simplicity of the notations we work with three rays. However Proposition \ref{main coro sect 5} holds for star shaped quivers with  any number of rais bigger than three.
 For such  quivers we consider  exceptional representations with hill dimension vector (Definition \ref{hill}) in addition to the already considered thin dimension vectors.  We show that for any two exceptional representations $\rho, \rho' \in Rep_k(Q)$, whose dimension vectors are hill or thin but not both hill, the map $F_{\rho,\rho'}^Q$ has maximal rank.  
The following lemma (more precisely  parts (a) and (b) in this lemma) is the first step to show this.
\begin{lemma} \label{big lemma} Let $L$ be a quiver whose  vertices are the numbers $\{1,2,\dots,n\}$($n\geq 2$), whose graph is  $\Gamma(L)=\bd[1em] 1 & \rLine & 2 &\rLine & \dots &\rLine & ( n-1 ) &\rLine & n \ed $, and  with any orientation of the arrows. Let $\rho, \rho' \in Rep_k(L)$ be  two representations  with dimension vectors $\alpha=\ul{\dim}(\rho)$, $\alpha'=\ul{\dim}(\rho')$, s. t. 
\begin{gather}   0<\alpha_1\leq \alpha_2 \leq \dots \leq \alpha_{n-1}\leq \alpha_n  \qquad \qquad 0<\alpha'_1\leq \alpha'_2 \leq \dots \leq \alpha'_{n-1}\leq \alpha'_n, \nonumber 
\end{gather}
and s. t. for each  $a\in Arr(L)$ the linear maps $\rho_a, \rho'_a$ have maximal rank.  Then the linear map 
\begin{gather} \label{F_rho,rho'linear} \bd \prod_{i=1}^n \Hom(k^{\alpha_i},k^{\alpha'_i}) & \rTo^{F^L_{\rho, \rho'} } & \prod_{a\in Arr(L)} \Hom(k^{\alpha_{s(a)}}, k^{\alpha'_{t(a)}})\ed  \quad 
 \{f_i\}_{i=1}^n \mapsto \left \{ 
f_{t(a)}\circ \rho_a - \rho'_a \circ f_{s(a)} \right \}_{ a \in Arr(L)} \end{gather}
has the following properties:

\textbf{(a)} The map $\ker(F^L_{\rho,\rho'}) \rightarrow \Hom(k^{\alpha_n},k^{\alpha'_n})$, defined by  $\ker(F^L_{\rho,\rho'}) \ni  \{f_i\}_{i=1}^n \mapsto f_n$, is injective.

\textbf{(b)} For any $x \in \Hom(k^{\alpha_1}, k^{\alpha'_1})$  and any $y\in \prod_{a\in Arr(L)} \Hom(k^{\alpha_{s(a)}}, k^{\alpha'_{t(a)}})$ there exists $\{f_i\}_{i=1}^n \in \prod_{i=1}^n \Hom(k^{\alpha_i},k^{\alpha'_i})$, s. t. $F^L_{\rho,\rho'}\left ( \{f_i\}_{i=1}^n  \right )=y$ and $f_1=x$. 

\textbf{(c)} In particular, the map $F^L_{\rho,\rho'}$ is surjective and  $\dim(\ker(F^L_{\rho,\rho'}))=\sum_{i=1}^n \alpha_i \alpha'_i - \sum_{a \in Arr(L)} \alpha_{s(a)} \alpha'_{t(a)} $. 

\textbf{(d)} If we are given a surjective map $\bd k^{\alpha'_0} & \lTo^{x}& k^{\alpha'_1} \ed$,  then the linear  map 
\begin{gather}  \bd \prod_{i=1}^n \Hom(k^{\alpha_i},k^{\alpha'_i}) & \rTo^{G^L_{\rho,\rho',x}} & \Hom(k^{\alpha_1},k^{\alpha'_0}) \oplus \prod_{a\in Arr(L)} \Hom(k^{\alpha_{s(a)}}, k^{\alpha'_{t(a)}})\ed   \nonumber\\[-5mm] \label{F_2}  \\ 
 \{f_i\}_{i=1}^n \mapsto (-x \circ f_1, F_{\rho,\rho'}^L(\{f_i\}_{i=1}^n)) \nonumber \end{gather}
is surjective and the dimension of its kernel is $\sum_{i=1}^n \alpha_i \alpha'_i - \sum_{a \in Arr(L)} \alpha_{s(a)} \alpha'_{t(a)}-\alpha_1 \alpha'_0 $. 

\textbf{(e)} If we are given an injective  map $\bd k^{\alpha_0} & \rTo^{x}& k^{\alpha_1} \ed$, then the linear  map 
\begin{gather}  \bd \prod_{i=1}^n \Hom(k^{\alpha_i},k^{\alpha'_i}) & \rTo^{H^L_{\rho,\rho',x}} & \Hom(k^{\alpha_0},k^{\alpha'_1}) \oplus \prod_{a\in Arr(L)} \Hom(k^{\alpha_{s(a)}}, k^{\alpha'_{t(a)}})\ed \nonumber  \\[-5mm] \label{F_3}\\
 \{f_i\}_{i=1}^n \mapsto (f_1 \circ x, F_{\rho,\rho'}^L(\{f_i\}_{i=1}^n)) \nonumber \end{gather}
is surjective and the dimension of its kernel is $\sum_{i=1}^n \alpha_i \alpha'_i - \sum_{a \in Arr(L)} \alpha_{s(a)} \alpha'_{t(a)}-\alpha_0 \alpha'_1 $. 

\textbf{(f)} If we are given an injective  map $\bd k^{\alpha'_n} & \rTo^{x}& k^{\alpha'_{n+1}} \ed$, then the linear  map 
\begin{gather}  \bd \prod_{i=1}^n \Hom(k^{\alpha_i},k^{\alpha'_i}) & \rTo & \Hom(k^{\alpha_n},k^{\alpha'_{n+1}}) \oplus \prod_{a\in Arr(L)} \Hom(k^{\alpha_{s(a)}}, k^{\alpha'_{t(a)}})\ed \nonumber  \\
 \{f_i\}_{i=1}^n \mapsto (-x \circ f_n, F_{\rho,\rho'}^L(\{f_i\}_{i=1}^n)) \nonumber \end{gather}
is injective. 
\end{lemma}
\bpr We prove first (a) and (b). Let $n=2$. We consider the two possible orientations of the arrow $\bd 1 & \rLine & 2\ed$. 

If the arrow starts at $1$, then consider the diagram \bd[1.5em]  k^{\alpha_1} & \rTo^{\rho} &  k^{\alpha_2} \\ 
                                                                    \dTo^{f_1}     &             & \dTo^{f_2}           \\
																																k^{\alpha'_1} & \rTo^{\rho'} &  k^{\alpha'_2}  \ed
Now the map $F_{\rho,\rho'}^L$ is $F_{\rho,\rho'}^L(f_1,f_2)=f_2 \circ \rho - \rho' \circ f_1 $ and $\rho$,  $\rho'$ are injective. 		If $	F_{\rho,\rho'}^L(f_1,f_2)=0$ and $f_2=0$, then 		$\rho' \circ f_1 =0$, and by the injectivity of $\rho'$ we obtain  $f_1 =0$. Thus, we obtain (a). 

To show (b) we have to find $f_2 \in \Hom(k^{\alpha_2}, k^{\alpha'_2})$, s. t. $f_2 \circ \rho - \rho' \circ x=y$ for any $x\in \Hom(k^{\alpha_1}, k^{\alpha'_1})$ and any $y\in \Hom(k^{\alpha_1}, k^{\alpha'_2})$. Since $\rho$ is injective, then it has left inverse $\pi:k^{\alpha_2} \rightarrow  k^{\alpha_1}$, and then we can choose $f_2=(y+\rho' \circ x)\circ \pi$.

If the arrow starts at $2$, then consider the diagram \bd[1.5em]  k^{\alpha_1} & \lTo^{\rho} &  k^{\alpha_2} \\ 
                                                                    \dTo^{f_1}     &             & \dTo^{f_2}           \\
																																k^{\alpha'_1} & \lTo^{\rho'} &  k^{\alpha'_2}  \ed
Now the map $F_{\rho,\rho'}^L$ is $F_{\rho,\rho'}^L(f_1,f_2)=f_1 \circ \rho - \rho' \circ f_2 $ and $\rho$,  $\rho'$ are surjective. 		If $	F_{\rho,\rho'}^L(f_1,f_2)=0$ and $f_2=0$, then 		$ f_1 \circ \rho =0$, and by the surjectivity  of $\rho$ we obtain  $f_1 =0$. Thus, we obtain (a). 

To show (b) we have to find $f_2 \in \Hom(k^{\alpha_2}, k^{\alpha'_2})$, s. t. $x \circ \rho - \rho' \circ f_2 =y$ for any $x\in \Hom(k^{\alpha_1}, k^{\alpha'_1})$ and any $y\in \Hom(k^{\alpha_2}, k^{\alpha'_1})$. Since $\rho'$ is surjective, then it has right inverse $in:k^{\alpha'_1} \rightarrow  k^{\alpha'_2}$, and then we can choose $f_2=in \circ (x \circ \rho-y)$.

So far, we proved  (a), (b), when $n=2$.  Now by using induction and the already proved  case $n=2$ one can easily prove (a), (b) for each $n\geq 2$. The statements in (c), (d), (e), and  (f) follow from (a) and (b).
\epr 
In Lemma \ref{second comp hill} we allow one of the components $\rho$, $\rho'$ to be of a  type different from thin. More precisely: 
\begin{df} \label{hill} Let $Q$ be a star shaped quiver(as in Figure \eqref{star shaped quiver}).  We say that  $\alpha \in \NN^{V(Q)}$ is a hill vector if  \begin{gather} \alpha(u_1) \leq \alpha(u_2) \leq \dots \leq \alpha(u_{m-1}) \leq \alpha(s), \qquad \alpha(u_{m-1}) > 0  \nonumber  \\
\label{nondecreasing} \alpha(v_1) \leq \alpha(v_2) \leq \dots \leq \alpha(v_{n-1}) \leq \alpha(s), \qquad \alpha(v_{n-1}) > 0     \\
\alpha(w_1) \leq \alpha(w_2) \leq \dots \leq \alpha(w_{p-1}) \leq \alpha(s), \qquad \alpha(w_{p-1}) > 0.  \nonumber   \end{gather}
In the case of number of rays bigger than three, we impose that $\alpha$ is non-vanishing on all the vertices adjacent to $s$.  
\end{df}

In the proof of  Lemma \ref{second comp hill}    we use the following  simple observation:
\begin{lemma} \label{lemma with dimensions} Let $Y\subset V$ be a vector subspace in a vector space $V$ and  $\dim(Y)=y$, $\dim(V)=n$. Let $\{x_1,\dots,x_m\}$ be  integers in $\{0,1,\dots, n\}$. 

(a) If $y+\sum_{i=1}^m x_i - m n \geq 0$, then one can choose vector supspaces $\{ X_i \subset V \}_{i=1}^m$ so that $\dim(X_i)=x_i$ and \be \dim\left (Y \cap \bigcap_{i=1}^m X_i \right ) = y+\sum_{i=1}^m x_i - m n\ee

(b) If $y+\sum_{i=1}^m x_i - m n < 0$, then one can choose vector supspaces $\{ X_i \subset V \}_{i=1}^m$ so that $\dim(X_i)=x_i$ and \be Y \cap \bigcap_{i=1}^m X_i  = \{0\}.\ee
\end{lemma}
\bpr \textbf{(a)} If $m=1$, then we have $y+x_1 \geq n$. Therefore we can choose $X_1\subset V$, so that $\dim(X_1)=x_1$ and $X_1+Y=V$. Therefore by a well known formula, we have $n=\dim(X_1+Y)=\dim(X_1)+\dim(Y)-\dim(X_1\cap Y)=$ $x_1+y - \dim(X_1\cap Y)$ . Hence $ \dim(X_1\cap Y)=x_1+y-n$.

Suppose that (a) holds for some $m\geq 1$ and take  any collection of integers $\{x_1,\dots,x_m, x_{m+1}\}$ in $\{0,1,\dots, n\}$, s. t. $y+\sum_{i=1}^{m+1} x_i - (m+1) n \geq 0$. We can rewrite the last inequality as follows
$ n \leq y+\sum_{i=1}^{m} x_i +x_{m+1} - m n$. On the other hand $x_{m+1}\leq n$, therefore 
\begin{gather}  n \leq y+\sum_{i=1}^{m} x_i +x_{m+1} - m n \leq  y+\sum_{i=1}^{m} x_i +n - m n \ \ \Rightarrow \ \ 0\leq  y+\sum_{i=1}^{m} x_i  - m n.\end{gather}
Now by the induction assumption we obtain vector subspaces $\{X_i\subset V\}_{i=1}^m$ with $\{ \dim(X_i)=x_i \}_{i=1}^m$ and $\dim \left (Y \cap \bigcap_{i=1}^m X_i\right ) = y + \sum_{i=1}^{m} x_i  - m n $. Now we have $\dim \left (Y \cap \bigcap_{i=1}^m X_i\right ) + x_{m+1} \geq n$ and as in the case $m=1$ we find a vector subspace $X_{m+1}\subset V$ with $\dim(X_{m+1})=x_{m+1}$ and $\dim \left (Y \cap \bigcap_{i=1}^{m+1} X_i\right )=\dim \left (Y \cap \bigcap_{i=1}^m X_i\right ) + x_{m+1}-n= y + \sum_{i=1}^{m+1} x_i  - (m+1) n$. Thus, we proved (a).

\textbf{(b)} If $m=1$, then the statement is obvious. Now we assume that we have (b) for some $m\geq 1$. Let $\{x_1,\dots,x_m, x_{m+1}\}$ be  any collection of integers in $\{0,1,\dots, n\}$, s. t. $y+\sum_{i=1}^{m+1} x_i - (m+1) n <0 $.  If $y+\sum_{i=1}^{m} x_i -m n <0$, then we use the induction assumption.  If $y+\sum_{i=1}^{m} x_i -m n \geq 0$, then we use (a) to obtain vector subspaces $\{X_i\subset V\}_{i=1}^m$ with $\{\dim(X_i)=x_i\}_{i=1}^m$ and  $\dim \left (Y \cap \bigcap_{i=1}^m X_i\right )=y+\sum_{i=1}^{m} x_i -m n$.  Now we have $\dim \left (Y \cap \bigcap_{i=1}^m X_i\right ) + x_{m+1} =y+\sum_{i=1}^{m+1} x_i -m n <n$, therefore we can choose $X_{m+1}\subset V$ with $\dim(X_{m+1})=x_{m+1}$ and $Y \cap \bigcap_{i=1}^{m+1} X_i = \{0\}$. The lemma is proved.
\epr
\begin{lemma} \label{second comp hill} Let $\rho, \rho' \in Rep_k(Q)$ be two exceptional representations with thin and hill dimension vectors, respectively. Then $F_{\rho, \rho'}$ and $F_{\rho',\rho}$ have maximal rank.
\end{lemma}
\bpr We show first that $F_{\rho,\rho'}$ has  maximal rank. Since now $\rho$ is exceptional and thin,  we can apply Remark \ref{thin remark} to it. In particular $\rho_a = 1$ for each $a \in Arr(A,A)$.   Due to Lemma \ref{lemma with one simple representation} we can assume that  $\#(A) \geq 2$, hence $A\cap A' \neq \{s\}$(we are given also that $\alpha'$ does not vanish on all vertices adjacent to $s$). Due to Lemma \ref{about A cap A' single element}  we can assume that $\#(A\cap A')\geq 2$. Lemma \ref{lemma about A setminus A'}   considers the case $Arr(A\setminus A', A'\setminus A) \not = \emptyset$, hence we can assume that  $Arr(A\setminus A', A'\setminus A)  = \emptyset$  and  we can write (recall \eqref{F_rho,rho'1}, \eqref{F_rho,rho'2},  \eqref{Arr(A,A')})
\begin{gather}\label{langle ranglex} 
\langle \alpha,\alpha' \rangle = \sum_{i\in A\cap A'}\alpha'_i- \sum_{a\in Arr(A\cap A',A\cap A')}  \alpha'_{t(a)}- \sum_{a\in Arr(A\cap A',A' \setminus A)}  \alpha'_{t(a)} -  \sum_{a\in Arr(A \setminus A',A\cap A')}  \alpha'_{t(a)} \\ \label{F_rho,rho'1x} F_{\rho, \rho'} : \prod_{i\in A\cap A'} \Hom\left (k,k^{\alpha'_i}\right ) \rightarrow \prod_{a\in Arr(A,A')} \Hom\left (k, k^{\alpha'_{t(a)}}\right )  \\ 
\label{F_rho,rho'2x} F_{\rho, \rho'}\left (\{f_i\}_{i\in A\cap A'} \right ) =\left \{  \begin{array}{cc}
f_{t(a)} - \rho'_a \circ f_{s(a)}  & a \in Arr(A\cap A',A\cap A')  \\
 - \rho'_a \circ f_{s(a)}  & a \in Arr(A\cap A',A' \setminus A)  \\
f_{t(a)}  & a \in Arr(A \setminus A',A\cap A')  \end{array} \right \}. \end{gather}

\ul{We consider first the case(see  \eqref{star shaped quiver}) $A\cap A'\subset \{u_1,u_2,\dots,u_{m-1}\}$.}  Now $\Gamma(Q_{A\cap A'}) $ has the form  \\ $\bd[1em] u_i & \rLine & u_{i+1} &\rLine & \dots &\rLine &  u_{i+k-1}  &\rLine & u_{i+k}  \ed $. Let us denote by $\rho_{r}$,  $\rho'_{r}$ the representations $\rho$, $\rho'$ restricted to $Q_{A\cap A'}$. Then $Q_{A\cap A'}$, $\rho_{r}$,  $\rho'_{r}$  satisfy the conditions in  Lemma \ref{big lemma} (recall also the last statement in Corollary \ref{main coro} ).
 Let us denote by $a$ the arrow adjacent to $Q_{A\cap A'}$ at  $u_{i+k}$. If $a$ starts at $u_{i+k}$, i. e. it points towards the splitting point $s$, then, due to  \eqref{nondecreasing}, $a\in Arr(A\cap A',A'\setminus A)$ and $\rho'_a$ is injective. In this case $\pi \circ F_{\rho,\rho'}$, where  $\pi$ is some projection, is the same as the linear map in Lemma \ref{big lemma} (f) with $x=\rho'_a$, hence  $F_{\rho,\rho'}$ is injective. Let the arrow $a$ ends at $u_{i+k}$, then it is neither in  $Arr(A\cap A',A'\setminus A)$ nor in  $Arr(A\setminus A',A'\cap A)$. 
Let us denote by $b$ the arrow  adjacent to $Q_{A\cap A'}$ at  $u_{i}$.
Now if $b$  starts at $u_{i}$ and $u_{i-1} \in A'$, then  $b\in Arr(A\cap A', A'\setminus A)$ and $F_{\rho,\rho'}$ is the same as the linear map $G^{Q_{A\cap A'}}_{\rho_r,\rho'_r}$ in Lemma \ref{big lemma} (d) with $x=\rho'_b$. If $b$ starts at $u_{i}$ and $u_{i-1} \not \in A'$, then  $ Arr(A\cap A', A'\setminus A)=$ $Arr(A\setminus A', A'\cap A)=\emptyset$ and $F_{\rho,\rho'}$ is the same as $F^{Q_{A\cap A'}}_{\rho_r,\rho'_r}$ from Lemma \ref{big lemma}.
If $b$  ends at $u_{i}$ and $u_{i-1} \in A$, then  $b\in Arr(A\setminus A', A'\cap A)$ and $F_{\rho,\rho'}$ is the same as the linear map $H^{Q_{A\cap A'}}_{\rho_r,\rho'_r}$ in Lemma \ref{big lemma} (e) with $x={\rm Id}_k$. If $b$ ends at $u_{i}$ and $u_{i-1} \not \in A$, then  $ Arr(A\cap A', A'\setminus A)=$ $Arr(A\setminus A', A'\cap A)=\emptyset$ and $F_{\rho,\rho'}$ is the same as $F^{Q_{A\cap A'}}_{\rho_r,\rho'_r}$ from Lemma \ref{big lemma}. Thus, we see that $F_{\rho,\rho'}$ has maximal rank, when  $A\cap A'\subset \{u_1,u_2,\dots,u_{m-1}\}$.

\ul{Next, we consider the case $A\cap A'\subset \{u_1,u_2,\dots,u_{m-1},s\}$ and $s\in A \cap A'$.}  Now $\Gamma(Q_{A\cap A'}) $ has the form  $\bd[1em] u_i & \rLine & u_{i+1} &\rLine & \dots &\rLine &  u_{m-1}  &\rLine & s  \ed $ and the vertices adjacent to $s$ different from $u_{m-1}$  ($v_{n-1}$, $w_{p-1}$ in figure \eqref{star shaped quiver} in case there are only three rays) are not elements of $A$.  We denote by $\rho_r$, $\rho'_r$ the restrictions of $\rho$, $\rho'$ to $Q_{A\cap A'}$. Let us denote   the set of arrows between $s$ and  the vertices adjacent to $s$ different from $u_{m-1}$ by $S'$. If all arrows  in $S'$ end    at $s$, then  none of them is in $Arr(A \setminus A',A\cap A')\cup Arr(A\cap A',A' \setminus A)$ and, as in the previous case, $F_{\rho,\rho'}$ is one of the following three linear maps: $F_{\rho_r, \rho'_r}^{Q_{A\cap A'}}$(see \eqref{F_rho,rho'linear}), $G_{\rho_r, \rho'_r,x}^{Q_{A\cap A'}}$(see \eqref{F_2}),  $H_{\rho_r, \rho'_r,{\rm Id}_k}^{Q_{A\cap A'}}$(see \eqref{F_3}) considered in Lemma \ref{big lemma}, hence $F_{\rho,\rho'}$ is surjective.  
It is useful to denote 
\begin{gather} S= S'\cap Arr(A\cap A',A' \setminus A). \end{gather} 
For $a\in S$ the linear map $\rho'_a$ is surjective and 
\begin{gather} \label{ker rho'_a} \forall a \in S \ \ \ \ \ \ \dim(\ker(\rho'_a))=\alpha'_s-\alpha'_{t(a)}.\end{gather}
Looking at \eqref{F_rho,rho'2x} we see that  $F_{\rho,\rho'}$ has the form 
\begin{gather} \label{F_rho...}F_{\rho, \rho'}\left (\{f_i\}_{i\in A\cap A'} \right ) = \left (T\left (\{f_i\}_{i\in A\cap A'} \right ), \left \{  
 - \rho'_a \circ f_{s}  \right \}_{ a \in S} \right ),\end{gather}
where $T$ is one of $F_{\rho_r, \rho'_r}^{Q_{A\cap A'}}$, $G_{\rho_r, \rho'_r,x}^{Q_{A\cap A'}}$,  $H_{\rho_r, \rho'_r,{\rm Id}_k}^{Q_{A\cap A'}}$. In the tree cases $\ker(F_{\rho, \rho'})\subset \ker(T)\subset \ker\left (F_{\rho_r, \rho'_r}^{Q_{A\cap A'}}\right )$ and by Lemma \ref{big lemma} (a) the linear map $\bd \ker(T) & \rTo^{\kappa} & \Hom(k,k^{\alpha_s})\ed$ defined by projecting to the $\Hom(k,k^{\alpha_s})$-component is injective(here and in \eqref{F_rho...} the notation  $s$ is the splitting vertex of the quiver $Q$ (as in figure \eqref{star shaped quiver}).  Now from \eqref{F_rho...} we see that 
\begin{gather} \label{dim(ker(F_rho,...))} \dim(\ker(F_{\rho, \rho'}))=\dim(\kappa(\ker(F_{\rho, \rho'})))=\dim\left (\kappa(\ker T) \cap \bigcap_{a \in S} \ker(\rho'_a) \right ) \end{gather}
On the other hand by and  (c), (d), (e) in Lemma \ref{big lemma}, \eqref{ker rho'_a}, and the formula \eqref{langle ranglex} one  shows that 
\begin{gather}\label{langle...}  \langle \alpha ,\alpha' \rangle = \langle \ul{\dim}(\rho),\ul{\dim}(\rho') \rangle = \dim(\kappa(\ker T)) + \sum_{a\in S} \dim( \ker(\rho'_a)) - \# (S) \  \alpha'_s. \end{gather}
 The feature of $\rho'$, due to the fact that it is an exceptional representation,   used so far is that  $\rho'_a$ is of maximal rank for any $a\in Arr(Q)$.  All considerations hold for any $\widetilde{\rho'} \in Rep(\alpha')$ s. t.  $\widetilde{\rho'}_a$ is of maximal rank for  $a\in Arr(Q)$. For any such $\widetilde{\rho'}$ we have $\kappa(\ker(\widetilde{T}))\subset \Hom(k,k^{\alpha'_s}) \cong  k^{\alpha'_s}$, $\ker(\widetilde{\rho'}_a)\subset \Hom(k,k^{\alpha'_s}) \cong k^{\alpha'_s}$ for $a\in S$  and  \eqref{dim(ker(F_rho,...))}, \eqref{langle...}  hold.  If $\langle \alpha ,\alpha' \rangle \geq 0$, then by Lemma \ref{lemma with dimensions} (a) we can choose $\{ \widetilde{\rho'}_a \}_{a \in S}$ (without changing the rest elements of $\rho'$) so that 
\be \langle \alpha ,\alpha' \rangle=  \dim(\kappa(\ker T)) + \sum_{a\in S} \dim( \ker(\widetilde{\rho'_a})) - \# (S)\  \alpha'_s =\dim\left (\kappa(\ker T) \cap \bigcap_{a \in S} \ker(\widetilde{\rho'}_a) \right ).  \ee  Therefore, by \eqref{dim(ker(F_rho,...))}   we get $\dim(\ker(F_{\rho,\widetilde{\rho'}}))= \langle \ul{\dim}(\rho),\ul{\dim}(\widetilde{\rho'}) \rangle$, which implies that $F_{\rho,\widetilde{\rho'}}$ is surjective(see Lemma \ref{lemma about F_rhi,rhi'} (b)). Now Corollary \ref{main coro} shows that $F_{\rho,\rho'}$ has maximal rank.
If $\langle \alpha ,\alpha' \rangle < 0$, then by Lemma \ref{lemma with dimensions} (b) we can choose $\{ \widetilde{\rho'}_a \}_{a \in S}$ (without changing the rest elements of $\rho'$) so that 
$ \{0\} = \kappa(\ker \wt{T}) \cap \bigcap_{a \in S} \ker(\widetilde{\rho'}_a) $. Hence \eqref{dim(ker(F_rho,...))}  implies that $F_{\rho,\widetilde{\rho'}}$ is injective. Now Corollary \ref{main coro} shows that $F_{\rho,\rho'}$ has maximal rank.

\ul{Finally, we consider the case $\{u_{m-1},s\} \subset A\cap A'\not \subset \{u_1,u_2,\dots,u_{m-1},s\}$.}
Now some of the vertices adjacent to  $s$ different from $u_{m-1}$ are in $A\cap A'$. 
Let $U$ be the set of vertices adjacent to $s$ which are in $ A\cap A'$, in particular $u_{m-1}\in U$. 
 Let us denote $L_u = Q_{A\cap A' \cap  \{u_1,u_2,\dots,u_{m-1},s\}}$, and let $\rho^u$, $\rho'^u$  be the restrictions of $\rho$, $\rho'$ to $L_u$. Similarly we obtain $L_j$,  $\rho^j$, $\rho'^j$ for any $j\in U$. Then we can apply Lemma \ref{big lemma} to $L_j$,  $\rho^j$, $\rho'^j$ for $j\in U$.  We denote  here  by $S'$  the set of arrows between $s$ and  the vertices adjacent to $s$ which are not in  $U$. Finally let:
\begin{gather} \label{S in the last case1}  S= S'\cap Arr(A\cap A',A' \setminus A). \end{gather}
For $a\in S$ the linear map $\rho'_a: k^{\alpha'_s}\rightarrow k^{\alpha'_{t(a)}}$ is surjective and 
\begin{gather} \label{ker rho'_a1} \forall a \in S \ \ \ \ \ \ \dim(\ker(\rho'_a))=\alpha'_s-\alpha'_{t(a)}\end{gather}
 Furthermore, using \eqref{F_rho,rho'2x} and Lemma \ref{big lemma}  we can express $F_{\rho,\rho'}$ as follows:
\begin{gather}  \label{F_rho...1} F_{\rho, \rho'}\left (\{f_i\}_{i\in A\cap A'} \right )  = \left (\left \{ T_j\left (\{f_i\}_{i\in V(L_j)} \right ) \right \}_{j\in U}, \left \{  
 - \rho'_a \circ f_{s}  \right \}_{ a \in S} \right ) 
 \end{gather}
where for $j\in U$ the linear map $T_j$ is one of $F_{\rho^j, \rho'^j}^{L_j}$(see \eqref{F_rho,rho'linear}), $G_{\rho^j, \rho'^j,x_j}^{L_j}$(see \eqref{F_2}),  $H_{\rho^j, \rho'^j,{\rm Id}_k}^{L_j}$(see \eqref{F_3}). Using (c), (d), (e) in  Lemma \ref{big lemma} and \eqref{ker rho'_a1},  \eqref{langle ranglex} one computes 
\begin{gather} \label{langle alpha,....} \langle \alpha, \alpha' \rangle = \sum_{j \in U} \dim(\ker T_j) + \sum_{a\in S} \dim( \ker(\rho'_a)) - (\# (U)+\# (S)-1 )\  \alpha'_s   
.\end{gather}
By Lemma \ref{big lemma} (a) the linear map $\bd \ker(T_j) & \rTo^{\kappa_j} & \Hom(k,k^{\alpha_s})\ed$ defined by projecting to the $\Hom(k,k^{\alpha_s})$-component is injective for $j\in U $. From \eqref{F_rho...1} it follows that   the linear map $\bd \ker(F_{\rho, \rho'}) & \rTo^{\kappa} & \Hom(k,k^{\alpha_s})\ed$ defined by projecting to the $\Hom(k,k^{\alpha_s})$-component is injective as well and  
\begin{gather} \label{dim(ker(F_rho,...))1} \dim(\ker(F_{\rho, \rho'}))= \dim(\kappa(\ker(F_{\rho, \rho'})))=\dim\left (\ \ \bigcap_{j\in U}\kappa_j(\ker T_j) \cap \bigcap_{a\in S} \ker (\rho'_a)\ \ \right ). \end{gather}
The obtained formulas hold for any $\wt{\rho'}\in Rep(\alpha')$ s. t. $\wt{\rho'}_a$ is of maximal rank for $a\in Arr(Q)$ (we denote the corresponding linear maps by $\wt{T}_j$). Due to \eqref{action of g}  and having that $\wt{T}_j$ is  $F_{\rho^j, \wt{\rho'}^j}^{L_j}$ or $G_{\rho^j, \wt{\rho'}^j,\wt{x}_j}^{L_j}$ or  $H_{\rho^j, \wt{\rho'}^j,{\rm Id}_k}^{L_j}$ we can move transitively  $ \kappa_j(\ker(\wt{T_j}))$ and $ \ker (\wt{\rho'}_a)$ inside $\Hom(k,k^{\alpha'_s}) \cong k^{\alpha'_s}$  by varying $\wt{\rho'}\in Rep(\alpha')$ (for $ \kappa_j(\ker(\wt{T_j}))$ this can be done by applying the action \eqref{left action} and the formulas \eqref{action of g}\footnote{one shows that changing $\rho'^j$ to $g'\cdot \rho'^j$ amounts to changing of $\kappa_j(\ker T_j)$ to $g'_s(\kappa_j(\ker T_j))$, where $g'\in \prod_{i\in V(L_j)} GL(\alpha_i,k)$}). Therefore, if $\langle \alpha, \alpha' \rangle\geq 0$, using \eqref{langle alpha,....} and  Lemma \ref{lemma with dimensions} (a),  we can ensure 
\be \label{help 1}  \dim\left (\ \ \bigcap_{j\in U}\kappa_j(\ker \wt{T}_j) \cap \bigcap_{a\in S} \ker (\wt{\rho'}_a)\ \ \right ) = \langle \alpha, \alpha' \rangle. \ee Therefore $\dim(\ker(F_{\rho, \wt{\rho'}}))=\langle \ul{\dim}(\rho), \ul{\dim}(\wt{\rho'}) \rangle$, which implies that $F_{\rho, \wt{\rho'}}$ is surjective(see Lemma \ref{lemma about F_rhi,rhi'} (b)). Now Corollary \ref{main coro} shows that $F_{\rho,\rho'}$ has maximal rank.
If $\langle \alpha ,\alpha' \rangle < 0$, then, due to Lemma \ref{lemma with dimensions} (b) and \eqref{langle alpha,....}, we can vary $\wt{\rho'}$  so that 
the lefthandside in \eqref{help 1} equals $ \{0\} $. Hence \eqref{dim(ker(F_rho,...))1}  implies that $F_{\rho,\widetilde{\rho'}}$ is injective. Now Corollary \ref{main coro} shows that $F_{\rho,\rho'}$ has maximal rank.

So far we proved that $F^Q_{\rho,\rho'}$ has a maximal rank. The quiver $Q^{\vee}$ and the representations $\rho^{\vee}$, $\rho'^{\vee}$ satisfy the same conditions as  $Q$, $\rho$, $\rho'$, respectively. Therefore $F^{Q^{\vee}}_{\rho^{\vee},\rho'^{\vee}}$ has maximal rank. Now Lemma \ref{lemma about F_rhi,rhi'} (e) shows that $F^Q_{\rho',\rho}$ has maximal rank. The lemma is proved.
\epr

 Combining Lemmas  \ref{second comp hill}  and \ref{lemma with units} we obtain:
\begin{prop} \label{main coro sect 5} Let $Q$ be a star shaped quiver. For any two exceptional representations $\rho, \rho' \in Rep_k(Q)$, whose dimension vectors are hill or thin and they are not both hill, the linear map $F_{\rho,\rho'}$ has maximal rank.  In particular, for any two such $\rho, \rho'$ we have $\hom(\rho,\rho')=0$ or $\hom^1(\rho,\rho')=0$. 
\end{prop}

\begin{coro} \label{tree arms} In any star quiver with three arms  $Q$  for any two exceptional representations $\rho, \rho'$ we  have  ${\rm Hom}(\rho,\rho')=0$ or ${\rm Ext}^1(\rho,\rho')=0$ provided that $\rho$ or $\rho'$ is a thin representation. 
\end{coro}
\bpr  From \cite[Corollary 1.4]{Ringelv3} it follows that   any exceptional representation in $Rep_k(Q)$ is either thin or has hill dimension vector. Therefore the corollary follows from Proposition \ref{main coro sect 5}. 
\epr

\section{Some  directions of future research}  \label{directions for future}
Motivated by   Corollaries \ref{noENCdynkin}, \ref{Dynkin are regularity preserv} and the results in \cite{DK} we conjecture 

\begin{conj} \label{conj1} For each Dynkin quiver $Q$ and each $\sigma \in \st(D^b(Q))$ there exists a full $\sigma$-exceptional collection.   
\end{conj}

When there are  Ext-nontrivial couples, RP property 1 and  RP property 2 are   our method to prove  regularity-preserving.
 The fact that they hold not only in $Rep_k(Q_1)$, but also in $Rep_k(Q_2)$ (Corollary \ref{RP property 1,2 and.. for Q1}) seems to be a trace of   a larger  unexplored  picture.  We expect  that there are further non-trivial examples of  regularity-preserving  categories. For example,  we conjecture
\begin{conj} For each quiver $Q$, whose underlying graph is a  loop, there is a choice of the orientation of the arrows such that the  category $Rep_k(Q)$ is regularity preserving. 
\end{conj}

 We showed in \cite{DK}  that $Rep_k(Q_2)$ is regularity-preserving, but did not answer the  question: is there  a $\sigma$-exceptional quadruple for each $\sigma \in \st(D^b(Q_2))$ (the $Q_2$-analogue of  Theorem \ref{main theorem for Q_1 in intro}).   The results of \cite[ Sections 7,8, Subsection 9.1]{DK}   hold for  $Rep_k(Q_2)$.  These  are clues for a positive answer. In section \cite[Section 2]{DK} are  given  the dimensions of $\Hom(X,Y)$, ${\rm Ext}^1(X,Y)$ for any two exceptional objects $X,Y$ in $Q_2$ as well. This lays  a  ground for working on the  $Q_2$-analogue of  Theorem \ref{main theorem for Q_1 in intro}.

\begin{quest} In \cite{DK} and in the present paper, the notion of regularity-preserving category was defined and applied to categories of homological dimension 1. To define and study a relevant notion  for higher dimensional categories  is another direction of future research.
\end{quest}

\end{document}